\documentclass{article}

\usepackage{amsfonts,amssymb,amsmath,amsthm,graphicx,authblk,rotating}
\usepackage{enumitem}
\usepackage{mathrsfs}
\usepackage{tikz}
\usepackage{pgfplots}
\pgfplotsset{compat=1.18}
\usepgfplotslibrary{fillbetween}
\usepackage{lscape}
\usepackage{cases}
\usepackage{mathabx}
\usepackage{lmodern}
\usepackage{upgreek,bm}
\usepackage[numbers]{natbib}

\newtheorem{theorem}{Theorem}
\newtheorem{lemma}{Lemma}
\newtheorem{proposition}{Proposition}
\newtheorem{remark}{Remark}
\usepackage{braket}
\usepackage[inter-unit-product=\cdot,
            separate-uncertainty=true,
            print-unity-mantissa=false,
            table-alignment-mode=none,
            table-alignment=right,
            exponent-product=\cdot]{siunitx}
\DeclareSIUnit\mmhg{mmHg}
\usepackage{hyperref}
\usepackage[capitalise]{cleveref}
    \crefname{proposition}{Proposition}{Propositions} 
    \Crefname{proposition}{Proposition}{Propositions}
    \crefname{lemma}{Lemma}{Lemmas} 
    \crefname{corollary}{Corollary}{Corollaries} 

\usepackage{cancel}
\usepackage{xstring}    

\makeatletter

\newcommand{\apost}{\emph{a posteriori}}
\newcommand{\Apost}{\emph{A posteriori}}

\newcommand{\PP}{{\rm el}}
\newcommand{\FF}{{\rm f}}
\let\vec\undefined
\DeclareRobustCommand{\vec}[1]{{\pmb{#1}}}
\newcommand{\dimens}{{\mathfrak D}}
\newcommand{\jj}{{\rm j}}
\newcommand{\kk}{{\rm k}}
\newcommand{\EE}{{\rm E}}

\newcommand{\II}{{\rm I}}

\newcommand{\test}[1]{{
        \widecheck{#1}}}
\newcommand{\mesh}[1]{{\mathscr{#1}}}
\newcommand{\elem}{{K}}
\newcommand{\face}{{F}}

\let\tmpsigma\sigma
\let\sigma\undefined
\newcommand{\sigel}{\tmpsigma_\PP}
\let\tmptau\tau
\let\tau\undefined
\newcommand{\tauf}{\tmptau_\FF}
\newcommand{\Div}[1]{{\rm div\,}#1}

\newcommand{\jumpl}{[\![}
\newcommand{\jumpr}{]\!]}
\newcommand{\Jumpl}{\left[\!\!\left[}
\newcommand{\Jumpr}{\right]\!\!\right]}

\newcommand{\jump}[1]{\jumpl#1\jumpr}
\newcommand{\Jump}[1]{\Jumpl#1\Jumpr}

\newcommand{\pedchn}[1]{%
    \IfSubStr{#1}{K}%
    {\StrBefore{#1}{K}[\tempA]%
        {\pedchn{\tempA},K}}%
    {\IfEqCase{#1}{%
        {c}{}%
        {h}{{ht}}%
        {n}{h}%
        {npp}{h}%
        {nmm}{h}
    }[\PackageError{pedchn}{Undefined option to pedchn: #1}{}]%
    }
}%
\newcommand{\apchn}[1]{%
    \IfSubStr{#1}{K}%
    {\StrBefore{#1}{K}[\tempA]%
        {\apchn{\tempA}}}%
    {\IfEqCase{#1}{%
        {0}{0}%
        {c}{}%
        {h}{}%
        {n}{n}%
        {npp}{{n+1}}%
        {nmm}{{n-1}}
    }[\PackageError{apchn}{Undefined option to apchn: #1}{}]}%
}%
\newcommand{\chn}[3][]{
    \IfStrEq{#1}{}{%
    #3^{\apchn{#2}}_{\pedchn{#2}}}%
    {   \IfSubStr{#1}{,}{%
        #3^{\apchn{#2}}_{#1\pedchn{#2}}}%
        {\chn[#1,]{#2}{#3}}
    }
}
\newcommand{\discrdt}[1]{    
    \delta_t^{\apchn{#1}}
}

\newcommand{\spHd}{{V_\vec{d}}}
\newcommand{\spLd}{{L_\vec{d}}}
\newcommand{\spHu}{{V_\vec{u}}}
\newcommand{\spLu}{{L_\vec{u}}}
\newcommand{\spHJ}{{V_J}}

\newcommand{\spHjj}{{V_\jj}}
\newcommand{\spLJ}{{L_J}}
\newcommand{\spLjj}{{L_\jj}}
\newcommand{\spLp}{{L_p}}
\newcommand{\spHstar}{{V_\star}}
\newcommand{\hspHd}{V_{\vec{d},h}}
\newcommand{\hspLd}{L_{\vec{d},h}}
\newcommand{\hspHu}{V_{\vec{u},h}}
\newcommand{\hspLu}{L_{\vec{u},h}}
\newcommand{\hspHJ}{V_{J,h}}
\newcommand{\hspHjj}{V_{\jj,h}}
\newcommand{\hspLJ}{L_{J,h}}

\newcommand{\hspLp}{L_{p,h}}
\newcommand{\hspHstar}{{V_{\star,h}}}
\newcommand{\DUALspHd}{{V'_\vec{d}}}
\newcommand{\DUALspHu}{{V'_\vec{u}}}
\newcommand{\DUALspHJ}{{V'_J}}
\newcommand{\DUALspLp}{{L'_p}}
\newcommand{\DUALspHstar}{{V'_\star}}
\newcommand{\bochner}[3]{{L^{#1}(0,#2; #3)}}
\newcommand{\bochnerH}[3]{{H^{#1}(0,#2; #3)}}
\newcommand{\Linf}[2]{%
    \IfEqCase{#2}{%
        {T}{\bochner{\infty}{T}{#1}}%
    }[\bochner{\infty}{t^{#2}}{#1}]%
}
\newcommand{\Ltwo}[2]{%
    \IfEqCase{#2}{%
        {T}{\bochner{2}{T}{#1}}%
    }[\bochner{2}{t^{#2}}{#1}]%
}
\newcommand{\Lone}[2]{%
    \IfEqCase{#2}{%
        {T}{\bochner{1}{T}{#1}}%
    }[\bochner{1}{t^{#2}}{#1}]%
}
\newcommand{\Hone}[2]{%
    \IfEqCase{#2}{%
        {T}{\bochnerH{1}{T}{#1}}%
    }[\bochnerH{1}{t^{#2}}{#1}]%
}

\newcommand{\interp}[2]{
    {I_{#1,h}{#2}}
}
\newcommand{\errint}[2]{
    {#2 - \interp{#1}{#2}}
}

\newcommand{\GG}{{\mathcal{G}}}
\newcommand{\LLstokes}{{\mathcal{L}}}
\newcommand{\err}[2]{%
    \IfStrEq{#1}{0}{%
    e_{#2}^{0}}%
    {   \IfStrEq{#1}{*}{%
            e_{#2}^{*}}%
        {   \IfSubStr{#1}{*}{%
                \StrBehind{#1}{*}[\tempB]%
                e_{#2}^{*,\apchn{\tempB}}}%
            {   e_{#2}^{\apchn{#1}}%
            }
        }
    }
}
\newcommand{\ERRnoR}{{\text{ERR}_e}}

\newcommand{\RR}[3][]{
    \IfStrEq{#1}{}{%
    R^{\apchn{#2}}_{\pedchn{#2},#3}}%
    {   \IfSubStr{#1}{,}{%
        R^{\apchn{#2}}_{#1\pedchn{#2},#3}}%
        {\RR[#1,]{#2}{#3}}
    }
}
\newcommand{\faceRR}[3][]{
    \IfStrEq{#1}{}{%
    S^{\apchn{#2}}_{\pedchn{#2},#3}}%
    {   \IfSubStr{#1}{,}{%
        S^{\apchn{#2}}_{#1\pedchn{#2},#3}}%
        {\faceRR[#1,]{#2}{#3}}
    }
}
\newcommand{\interffaceRR}[3][]{
    \IfStrEq{#1}{}{%
    {S^{\apchn{#2}}_{\Sigma,\pedchn{#2},#3}}}%
    {   \IfSubStr{#1}{,}{%
        {S^{\apchn{#2}}_{\Sigma,#1\pedchn{#2},#3}}}%
        {\interffaceRR[#1,]{#2}{#3}}
    }
}
\newcommand{\PREinterffaceRR}[3][]{
    \overline{\interffaceRR[#1]{#2}{#3}}
}
\newcommand{\errRR}[3][]{
    \IfStrEq{#1}{}{%
    \mathscr{R}^{\apchn{#2}}_{\pedchn{#2}#3}}%
    {   \IfSubStr{#1}{,}{%
        \mathscr{R}^{\apchn{#2}}_{#1\pedchn{#2}#3}}%
        {\RR[#1,]{#2}{#3}}
    }
}

\newcommand{\estim}{{\mathcal{E}}}
\newcommand{\esttime}[1]{
    \IfStrEq{#1}{T}{%
        {\esttime{N_T}}}%
    {\estim_{\rm time}^{#1}}%
}
\newcommand{\estdata}[1]{
    \IfStrEq{#1}{T}{%
        {\estdata{N_T}}}%
    {\estim_{\rm data}^{#1}}%
}
\newcommand{\tildeestdata}[1]{
    \IfStrEq{#1}{T}{%
        {\tildeestdata{N_T}}}%
    {\widetilde{\estim}_{\rm data}^{#1}}%
}
\newcommand{\estloc}[1]{
    \IfStrEq{#1}{T}{%
        {\estloc{N_T}}}%
    {\widehat{\estim}^{#1}}%
}
\newcommand{\estpre}[1]{
    \IfStrEq{#1}{T}{%
        {\estpre{N_T}}}%
    {\estim_{\GG}^{#1}}%
}
\newcommand{\eststokes}[1]{
    \IfStrEq{#1}{T}{%
        {\eststokes{N_T}}}%
    {\estim_{\LLstokes}^{#1}}%
}
\newcommand{\estok}[1]{
    \IfStrEq{#1}{T}{%
        {\estok{N_T}}}%
    {\estim_{\rm spc}^{#1}}%
}

\makeatother

\title{A posteriori error analysis for a coupled Stokes-poroelastic system with multiple compartments}
\author{Ivan Fumagalli, Nicola Parolini, Marco Verani
}
\date{}
\affil{\small MOX, Dipartimento di Matematica, Politecnico di Milano,\\
              piazza Leonardo da Vinci 32, 20133, Milan, Italy}
\begin{document}

\maketitle

\begin{abstract}
\noindent The discretization of fluid-poromechanics systems is typically highly demanding in terms of computational effort.
This is particularly true for models of multiphysics flows in the brain, due to the geometrical complexity of the cerebral anatomy -- requiring a very fine computational mesh for finite element discretization -- and to the high number of variables involved.
Indeed, this kind of problems can be modeled by a coupled system encompassing the Stokes equations for the cerebrospinal fluid in the brain ventricles and Multiple-network Poro-Elasticity (MPE) equations describing the brain tissue, the interstitial fluid, and the blood vascular networks at different space scales.
The present work aims to rigorously derive \apost{} error estimates for the coupled Stokes-MPE problem, as a first step towards the design of adaptive refinement strategies or reduced order models to decrease the computational demand of the problem.
Through numerical experiments, we verify the reliability and optimal efficiency of the proposed \apost{} estimator and identify the role of the different solution variables in its composition.
\end{abstract}

\section{Introduction}\label{intro}

The numerical modeling of multiphysics flows in the human brain poses several difficulties, due to the complexity of the brain's geometry and the computational cost of handling several coupled physical systems.
This modeling is of paramount importance in the investigation of the Cerebrospinal Fluid (CSF), whose main functions are to wash out the waste products of cerebral activity and protect the brain from impact with the skull \cite{glymphatic2,sakka2011anatomy}.
The CSF is generated in the cerebral tissue by mass exchange through the walls of capillary blood vessels and permeates the whole organ in its interstitial space: the interaction between these fluid networks and the elastic tissue can be modeled by Multiple-network Poro-Elasticity (MPE) equations \cite{chou2016fully,corti2023numerical,guo2018subject,lee2019mixed}.
This system is then coupled with the CSF flowing in the hollow cavities of the cerebral ventricles and the subarachnoid spaces, where the CSF flow can be modeled by Stokes equations \cite{fumagalli2024polytopal,gholampour2023mathematical,linninger2016cerebrospinal}.

The large number of variables encompassed by 
fluid\nobreakdash-poromechanics
models and the geometrical complexity of the brain and fluid-tissue interface make the numerical simulation of the problem particularly demanding.
To reduce the computational effort, different strategies can be considered:
adaptive refinement allows for retaining geometric accuracy while decreasing the computational demands, while reduced-order models provide an efficient means to approximate the solution of complex problems for different values of the model parameters \cite{hesthaven2016certified,plewa2005adaptive,quarteroni2015reduced,verfurth1994posteriori}.
Both these strategies are classically based on \apost{} error estimates, which have been derived for several single-physics problems, including single-fluid Biot equations \cite{ahmed2019adaptive,li2022residual,riedlbeck2017stress} or (Navier-)Stokes equations \cite{ainsworth1997posteriori,bank1991posteriori,hannukainen2012unified,verfurth1989posteriori}.
An \apost{} analysis of coupled Biot-Stokes system has been carried out in \cite{babuvska2010residual,houedanou2022posterioriStokesBiotNonconforming,houedanou2021posterioriStokesBiotLagrange}, but the case of multiple interacting fluids is scarcely covered by the \apost{} literature: the MPE problem alone has been addressed only in \cite{nordbotten2010posteriori} for a particular case and in the recent work \cite{eliseussen2023posteriori}, but the coupled MPE-Stokes problem with time-dependent pressure equations seems to be missing.

The present work aims at filling this gap, providing rigorous \apost{} estimates for the coupled MPE-Stokes problem with time-dependent pressure equations.
Specifically, we provide reliable residual-based estimators in the abstract framework of \cite{ern2009posteriori}, enhanced to account for multi-domain problems.
Moreover, through numerical experiments, we analyze the efficiency of these estimators and assess their main components.

Starting from \cref{sec:model}, we introduce the MPE-Stokes problem and its discretization by finite elements in space and the implicit Euler scheme in time.
Then, \cref{sec:apost} is devoted to the derivation of \apost{} error estimates for the solution to the problem.
In \cref{sec:results} we analyze the reliability and efficiency of the estimators and discuss the relevance of their main components.

\section{The coupled Stokes-MPE system}\label{sec:model}

\begin{figure}
    \centering
    \includegraphics[width=0.5\textwidth]{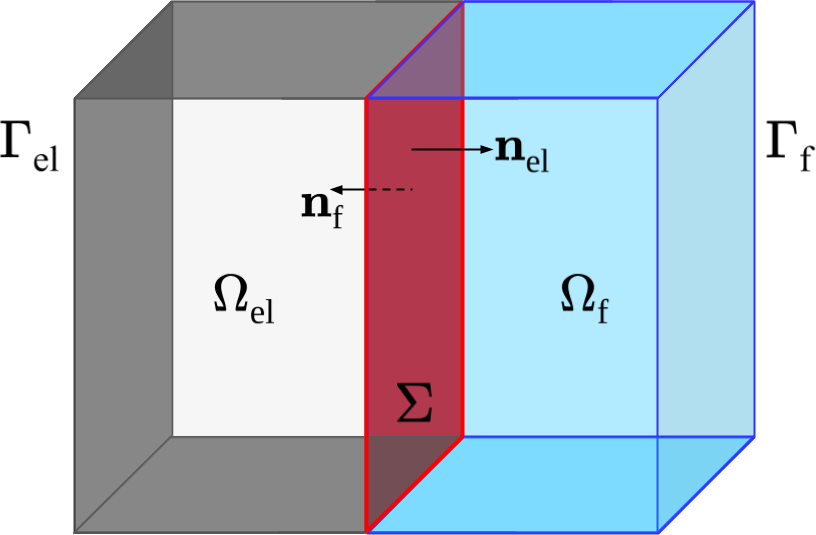}
    \caption{Domain scheme: poroelastic domain $\Omega_\PP$ (light grey), Stokes' domain $\Omega_\FF$ (blue), interface $\Sigma$ (red), and external boundaries $\Gamma_\PP=\Gamma_{\text{D},\vec{d}}\cup\Gamma_{\text{N},\vec{d}}=\Gamma_{\text{D},J}\cup\Gamma_{\text{N},J}$ and $\Gamma_\FF=\Gamma_{\text{D},\vec{u}}\cup\Gamma_{\text{N},\vec{u}}$.}
    \label{fig:domain}
\end{figure}

Let us consider a polyhedral $\dimens$-dimensional domain $\Omega\subset\mathbb R^\dimens$ ($\dimens=2,3$), schematically represented in \cref{fig:domain}, partitioned into a poroelastic region $\Omega_\PP$ and a fluid region $\Omega_\FF$ by an interface $\Sigma=\partial\Omega_\PP\cap\partial\Omega_\FF$ composed of a finite number of flat polygons.
Stokes equations are set in $\Omega_\FF$, with $\vec{u}$ and $p$ denoting the fluid's velocity and pressure, while $\Omega_\PP$ is filled with a linear poroelastic medium subject to the MPE equations, with solid displacement $\vec{d}$ and network pressures $p_\jj\in J$, where $J$ is a set of $\#J$ indices.
The external boundaries $\Gamma_\PP=\partial\Omega_\PP\setminus\Sigma, \Gamma_\FF=\partial\Omega_\FF\setminus\Sigma$ are partitioned in Dirichlet and Neumann portions for the different variables, with clear notation:
$\Gamma_\PP = \Gamma_{\text{D},\vec{d}}\cup\Gamma_{\text{N},\vec{d}}=\Gamma_{\text{D},J}\cup\Gamma_{\text{N},J}, \Gamma_\FF=\Gamma_{\text{D},\vec{u}}\cup\Gamma_{\text{N},\vec{u}}$;
notice that, for simplicity, we consider the same Dirichlet/Neumann splitting of the boundaries for all fluid networks $\jj\in J$.
We consider the Stokes and linear elasticity equations to be steady, while the time dependence of the porous flow is accounted for in the porous fluid momentum equations, as follows:
\begin{subnumcases}{\label{eq:NSMPE}}
    - \nabla\cdot\sigel(\vec{d}) + \sum_{\kk\in J}\alpha_\kk\nabla p_\kk = \vec{f}_\PP,
    &$ 
    \text{in } \Omega_\PP\times(0,T],$ \label{eq:elasticity}\\
    c_\jj\partial_t p_\jj+\nabla\cdot\left(\alpha_\jj\partial_t\vec{d}-\frac{\kappa_\jj}{\mu_\jj}\nabla p_\jj\right)
    \\\qquad
    + \sum_{\kk\in J}\beta_{\jj\kk}(p_\jj-p_\kk) + \beta_\jj^\text{e}p_\jj = g_\jj,
    &$ 
    \text{in } \Omega_\PP\times(0,T],\quad\forall\jj\in J,$ \label{eq:pj}\\
    - \nabla\cdot\tauf(\vec{u}) + \nabla p = \vec{f}_\FF,
    &$ 
    \text{in } \Omega_\FF\times(0,T],$ \label{eq:fluidMom}\\
    \nabla\cdot\vec{u} = 0,
    &$ 
    \text{in } \Omega_\FF\times(0,T],$ \label{eq:fluidCont}\\
    \vec{d}=\vec{d}_0, \quad p_\jj=p_{\jj,0},
    &$ 
    \text{in } \Omega_\PP\times\{0\},\quad\forall\jj\in J,$ \label{eq:ICel}\\
    \vec{u}=\vec{u}_0,
    &$ 
    \text{in } \Omega_\FF\times\{0\},$ \label{eq:ICf}\\
    \star = \vec{0},
    &$ 
    \text{on } \Gamma_{\text{D},\star}\times(0,T],\text{ for }\star\in\{\vec{d},\vec{u}\}$ \label{eq:dirichletDU}\\
    p_\jj = p_{\jj0},
    &$ 
    \text{on } \Gamma_{\text{D},J}\times(0,T], \quad \forall\jj\in J$ \label{eq:dirichletJ}\\
    \sigel(\vec{d})\vec{n}+\sum_{\kk\in J}\alpha_\kk p_\kk\vec{n} = 0
    &$ 
    \text{on } \Gamma_{\text{N},\vec{d}}\times(0,T],$ \label{eq:neumannD}\\
    \frac{\kappa_\jj}{\mu_\jj}\nabla p_\jj\vec{n} = 0
    &$ 
    \text{on } \Gamma_{\text{N},J}\times(0,T], \quad \forall\jj\in J$ \label{eq:neumannJ}\\
    \tauf(\vec{u})\vec{n}+p\vec{n} = 0
    &$ 
    \text{on } \Gamma_{\text{N},\vec{u}}\times(0,T],$ \label{eq:neumannU}\\
    \text{interface conditions (see \eqref{eq:interf} below)},
    &$ 
    \text{on } \Sigma\times(0,T],$
\end{subnumcases}
where $\sigel(\vec{d})=2\mu_\PP\varepsilon(\vec{d}) + \lambda\Div{\vec{d}}I$ is the Cauchy stress tensor of the elastic medium and $\tauf(\vec{u})=2\mu_\FF\varepsilon(\vec{u})$ is the viscous stress tensor of the fluid -- with $\varepsilon(\vec{\phi})=\frac{1}{2}(\nabla\vec{\phi}+\nabla\vec{\phi}^T)$.
We assume that, among all the fluid networks indexed in $J$, only one exchanges mass at the interface $\Sigma$ with the Stokes domain $\Omega_\FF$ and we denote it by $\EE$: the others exchange mass only among themselves and with $\EE$ (cf.~the terms with $\beta_{\jj\kk}$ in \eqref{eq:pj} and \cref{rmrk:brain} below).
Accordingly, the following interface conditions are imposed on $\Sigma$:
\begin{subnumcases}{\label{eq:interf}}
    \sigel(\vec{d})\vec{n}_\PP - \sum_{\kk\in J} \alpha_\kk p_\kk\vec{n}_\PP + \tauf(\vec{u})\vec{n}_\FF - p \vec{n}_\FF = \vec{0}, 
    &$ 
    \text{on } \Sigma\times(0,T],$
    \label{eq:BCtotalstress}\\
    p_\EE = p - \tauf(\vec{u})\vec{n}_\FF\cdot\vec{n}_\FF, 
    &$ 
    \text{on } \Sigma\times(0,T],$\label{eq:BCnormalstress}\\
    \frac{\kappa_\jj}{\mu_\jj}\nabla p_\jj\cdot\vec{n}_\PP = 0,
    \quad\forall\jj\in J\setminus\{\EE\},
    &$ 
    \text{on } \Sigma\times(0,T],
    ,$\label{eq:BCpnonE}\\
    \vec{u}\cdot\vec{n}_\FF + \left(
    \partial_t\vec{d}
    -\frac{\kappa_\EE}{\mu_\EE}\nabla p_\EE\right)\cdot\vec{n}_\PP = 0,
    &$ 
    \text{on } \Sigma\times(0,T],$\label{eq:BCnormalflux}\\
    \left(\tauf(\vec{u})\vec{n}_\FF
    \right)_{\parallel} = 
    \vec{0},
    &$ 
    \text{on } \Sigma\times(0,T],$\label{eq:BCtgstress}
\end{subnumcases}
where $(\vec{\phi})_\parallel=\vec{\phi}-(\vec{\phi}\cdot\vec{n}_\FF)\vec{n}_\FF$ is the tangential component of $\vec{\phi}$ along $\Sigma$.

\begin{remark}[Specifics of brain modeling]\label{rmrk:brain}
In the case of brain multiphysics flow, the fluid networks index set $J=\{\text{A},\text{V},\text{C},\EE\}$ contains the arterial (A), venous (V) and capillary (C) blood flow and the extracellular/interstitial cerebrospinal fluid ($\EE$).
Due to the brain-blood barrier \cite{abbott2010structure}, no direct flow of the blood compartments A, V, C occurs through the interface $\Sigma$: only the interstitial CSF pours into the CSF domain $\Omega_\FF$.
Regarding the choice of considering a quasi-static approximation for the elasticity and Stokes equations in \eqref{eq:NSMPE}, the resulting system can be employed to study different conditions:
\begin{itemize}
    \item the development of pathologies like hydrocephalus, with time scales of days or weeks that are significantly longer than the tissue relaxation times \cite{tully2011cerebral};
    \item the pulsatile CSF flow in the scale of seconds, if a slight overestimation of the ventricular wall displacements and intracranial pressure is admissible \cite{chou2016fully}.
\end{itemize}
\end{remark}

\subsection{Variational formulation}\label{sec:weak}

We introduce the following Sobolev spaces, with $\jj\in J$:
\begin{equation}\label{eq:spaces}\begin{aligned}
    \spHjj &= H^1_{\Gamma_{\text{D},J}}(\Omega_\PP)
    &
    \spHd &= [H^1_{\Gamma_{\text{D},\vec{d}}}(\Omega_\PP)]^\dimens, &
    \spHu &= [H^1_{\Gamma_{\text{D},\vec{u}}}(\Omega_\FF)]^\dimens, \\
    \spLjj &= L^2(\Omega_\PP),
    &
    \spLd &= [L^2(\Omega_\PP)]^\dimens, &
    \spLu &= [L^2(\Omega_\FF)]^\dimens,\\
    \spHJ &= [\spHjj]^{\#J}, &
    \spLJ & =[\spLjj]^{\#J}, &
    \spLp &= L^2(\Omega_\FF).
\end{aligned}\end{equation}
To simplify the notation, we employ the notation $\vec{p}_J=[p_\jj]_{\jj\in J}$ to indicate the elements of $\spHJ$ and $\spLJ$: these functions are vector fields of dimension $\#J$ having $p_\EE$ as their last component.
To account for the time dependence of the system variables, we set up our problem in the Bochner spaces $\Hone{\mathscr{V}}{T}$, where $\mathscr{V}$ is any Hilbert space introduced in \eqref{eq:spaces}.
With analogous notation, we will also consider the spaces $\Ltwo{\mathscr{V}}{T}, \Lone{\mathscr{V}}{T}, \Linf{\mathscr{V}}{T}$.

The variational formulation of problem \eqref{eq:NSMPE} reads as follows:
\\
Find $(\vec{d},\vec{p}_J,\vec{u},p)\in \Hone{\spHd\times\spHJ\times\spHu\times\spLp}{T}$ 
such that, for a.e. $t\in[0,T]$,
\begin{subnumcases}{\label{eq:weak}}
    a_\PP(\vec{d},\test{\vec{d}}) + b_J(\vec{p}_J,\test{\vec{d}})
    + \mathfrak J_\PP(p_\EE,\test{\vec{d}}) = F_\PP(\test{\vec{d}})
    &\\
    \begin{aligned}
        &m_J(\partial_t \vec{p}_J,\test{\vec{p}}_J)
    + \widetilde{a}_J(\vec{p}_J,\test{\vec{p}}_J)
    - b_J(\test{\vec{p}}_J,\partial_t\vec{d})
    \\&\qquad\qquad
    - \mathfrak J_\PP(\test{p}_\EE,\partial_t\vec{d}) - \mathfrak J_\FF(\test{p}_\EE,\vec{u})
    = F_J(\test{\vec{p}}_J)
    \end{aligned}
    &\\
    a_\FF(\vec{u},\test{\vec{v}})
    + b_\FF(p,\test{\vec{u}})
    + \mathfrak J_\FF(\test{p}_\EE,\vec{u})
    = F_\FF(\test{\vec{v}})
        \label{eq:weak-stokesMom}&\\
    b_\FF(\test{p},\vec{u}) = 0
        \label{eq:weak-stokesCont}&
\end{subnumcases}
for all $\test{\vec{d}}\in\spHd, \test{\vec{p}}_J\in\spHJ, \test{\vec{u}}\in\spHu, \test{p}\in\spLp$,
where
\begin{equation*}
\begin{gathered}
a_\PP(\vec{d},\test{\vec{d}}) =
    (\sigel(\vec{d}),\varepsilon(\test{\vec{d}}))_{\Omega_\PP},
\qquad\qquad
a_\FF(\vec{u},\vec{v}) =
    (\tauf(\vec{u}),\varepsilon(\test{\vec{u}}))_{\Omega_\FF},
\\
m_J(\partial_t\vec{p}_J,\test{\vec{p}}_J) =
    \sum_{\jj\in J}\left(c_\jj \partial_t p_\jj, \test{p}_\jj\right)_{\Omega_\PP},
\\
\widetilde{a}_J(\vec{p}_J,\test{\vec{p}}_J) = 
    \sum_{\jj\in J}
    \left[
    \left(\frac{\kappa_\jj}{\mu_\jj}\nabla p_\jj,\nabla \test{p}_\jj\right)_{\Omega_\PP}
    +
    (\beta_{\jj}^\text{e}p_\jj, \test{p}_\jj)_{\Omega_\PP} + \sum_{\kk\in J}(\beta_{\kk\jj}(p_\jj-p_\kk), \test{p}_\jj)_{\Omega_\PP}\right],
\\
b_J(\test{\vec{p}}_J,\test{\vec{d}}) = 
    -\sum_{\jj\in J}
    (\alpha_\jj \test{p}_\jj,\Div\test{\vec{d}})_{\Omega_\PP},
\qquad\qquad
b_\FF(\test{p},\test{\vec{u}}) =
    -(\test{p},\Div\test{\vec{u}})_{\Omega_\FF},
\\
\mathfrak J_\star(\test{p}_\EE,\vec{\phi}) =
    \int_\Sigma \test{p}_\EE\vec{\phi}\cdot\vec{n}_\star,
    \quad \star\in\{\PP,\FF\},\\
F_\PP(\test{\vec{d}}) = (\vec{f}_\PP,\test{\vec{d}})_{\Omega_\PP},
\qquad
G_J(\test{\vec{p}}_J) = \sum_{\jj\in J}(g_\jj,\test{p}_\jj)_{\Omega_\PP},
\qquad
F_\FF(\test{\vec{u}}) = (\vec{f}_\FF,\test{\vec{u}})_{\Omega_\FF}.
\end{gathered}\end{equation*}
In these definitions, $(\cdot,\cdot)_D$ denotes the $L^2$ product over a domain $D$.
Analogously, $\|\cdot\|_D$ will denote the norm of $L^2(D)$.


\subsection{Time and space discretization}

We introduce a simplicial mesh $\mesh{T}$ partitioning the whole domain $\Omega$, and we split it into two submeshes $\mesh{T}_\PP,\mesh{T}_\FF$, corresponding to the subdomains $\Omega_\PP,\Omega_\FF$ and conforming with the interface $\Sigma$.
The sets of codimension-1 internal facets are denoted by $\mesh{F}_\PP^I,\mesh{F}_\FF^I$ (triangles for $\dimens=3$, line segments for $\dimens=2$).
Analogously, we denote by $\mesh{F}_\Sigma$ the facets lying on the interface $\Sigma$ and  by $\mesh{F}_{\text{D},\phi}$ the facets lying on the corresponding Dirichlet boundary $\Gamma_{\text{D},\phi}, \phi\in\{\vec{d},\vec{p}_J,\vec{u}\}$.
In the following, we will denote by $h_\elem$ the characteristic size of an element $\elem\in\mesh{T}$.

On this partitioning, we introduce the following conforming finite element spaces:
\begin{equation}\label{eq:fe}
\begin{gathered}
X_\star^r=\{\phi_h\in C^0(\overline{\Omega}_\star) \colon \phi_h|_\elem\in \mathbb P^r(\elem)\ \forall \elem\in\mesh{T}\}, \quad \star\in\{\PP,\FF\},\\
\hspHd = \spHd\cap [X_\PP^{s+1}]^\dimens, \qquad
\hspHJ = \spHJ\cap [X_\PP^{s+1}]^{\#J}, \\
\hspLd = \spLd\cap [X_\PP^{s+1}]^\dimens, \qquad
\hspLJ = \spLJ\cap [X_\PP^{s+1}]^{\#J}, \\
\hspHu = \spHu\cap [X_\PP^{s+1}]^\dimens, \qquad
\hspLu = \spLu\cap [X_\PP^{s+1}]^\dimens, \qquad
\hspLp = \spLp\cap X_\PP^s.
\end{gathered}
\end{equation}

We denote by $\interp{\vec{d}}{}\colon\spHd\to\hspHd,\interp{J}{}\colon\spHJ\to\hspHJ,\interp{\vec{u}}\colon\spHu\to\hspHu$ suitable interpolation operators onto the discrete spaces introduced above, such that the following interpolation estimate holds:
\begin{equation}\label{eq:estInterp}
\sum_{\elem\in\mesh{T}}\left[h_\elem^{-2}\|\test{\vec{d}}-\interp{\vec{d}}{\test{\vec{d}}}\|_\elem^2 + h_\elem^{-1}\|\test{\vec{d}}-\interp{\vec{d}}{\test{\vec{d}}}\|_{\partial\elem}^2\right] \lesssim \|\test{\vec{d}}\|_\spHd^2 \qquad \forall \test{\vec{d}}\in\spHd,
\end{equation}
and analogous ones for $\interp{J}{},\interp{\vec{u}}{}$.
For example, Cl\'ement interpolators can be employed \cite{ern2009posteriori}.
We anticipate that no interpolation operator is needed over the Stokes pressure space $\spLp$.

Regarding time discretization, we consider an implicit Euler scheme on a uniform time grid made of $N_T$ subintervals $I^{\apchn{n}}=[t^{\apchn{nmm}},t^{\apchn{n}}], \apchn{n}=1,\ldots,N_T,$ of length $\Delta t$, with $t^0=0, N_T=T/{\Delta t}$.
All the \apost{} results presented hereafter can be easily extended to nonuniform time discretization, as done, e.g., in \cite{ern2009posteriori,eliseussen2023posteriori,verfurth1989posteriori}.

In the following, for any function $\chn{c}{\phi}(t,\vec{x})$ we will use the notation $\chn{n}{\phi}$ to indicate its full space-time discretization evaluated at time $t^{\apchn{n}}$, whereas $\chn{h}{\phi}$ will denote the continuous piecewise linear (in time) function such that $\chn{h}{\phi}|_{t=t^{\apchn{n}}}=\chn{n}{\phi}$.
We also introduce the projector $\pi^0$ onto piecewise constant functions in time, such that $\pi^0\chn{h}{\phi}|_{t\in(t^{\apchn{nmm}},t^{\apchn{n}}]}=\chn{n}{\phi}$.

Therefore, the fully discrete problem approximating \eqref{eq:weak} reads as follows:
\\
Find $(\chn{h}{\vec{d}},\chn[J]{h}{\vec{p}},\chn{h}{\vec{u}},\chn{h}{p})\in C^0([0,T]; \spHd\times\spHJ\times\spHu\times\spLp)$, piecewise linear in time, 
such that, for a.e. $t\in[0,T]$,
\begin{subnumcases}{\label{eq:discr}}
    a_\PP(\chn{h}{\vec{d}},\test{\vec{d}}_h) + b_J(\chn[J]{h}{\vec{p}},\test{\vec{d}}_h)
    + \mathfrak{J}_{\PP}(\chn[\EE]{h}{p},\test{\vec{d}}_h) = (
    \chn[\vec{d}]{h}{\vec{f}}
    ,\test{\vec{d}}_h)_{\Omega_{\PP}}
    \\
    m_J(\partial_t \chn[J]{h}{\vec{p}},\test{\vec{p}}_{J,h})
    + \widetilde{a}_J(\pi^0\chn[J]{h}{\vec{p}},\test{\vec{p}}_{J,h})
    - b_J(\test{\vec{p}}_{J,h},\partial_t\chn{h}{\vec{d}})\\
    \qquad\qquad- \mathfrak J_\PP(\test{p}_{\EE,h},\partial_t\chn{h}{\vec{d}}) - \mathfrak J_\FF(\test{p}_{\EE,h},\chn{h}{\vec{u}})
    = (\pi^0\chn{h}{\vec{g}},\test{\vec{p}}_{J,h})_{\Omega_\PP}
    \\
    a_\FF(\chn{h}{\vec{u}},\test{\vec{u}}_h)
    + b_\FF(\chn{h}{p},\test{\vec{u}}_h)
    + \mathfrak J_\FF(\chn[\EE]{h}{p},\test{\vec{u}}_h)
    = (
    \chn[\vec{u}]{h}{\vec{f}}
    ,\test{\vec{u}}_h)_{\Omega_\FF}
    \\
    b_\FF(\test{p}_h,\chn{h}{\vec{u}}) = 0
\end{subnumcases}
for all $\test{\vec{d}}_h\in\hspHd, \test{\vec{p}}_{J,h}\in\hspHJ, \test{\vec{u}}_h\in\hspHu, \test{p}_h\in\hspLp$,
and $\chn{h}{\vec{d}}|_{t=0}\simeq\vec{d}_0,
    \chn[\jj]{h}{p}|_{t=0}\simeq p_{\jj0}\,\forall \jj\in J,
    \chn{h}{\vec{u}}|_{t=0}\simeq \vec{u}_0,
$.
Notice that $\partial_t\chn[J]{h}{\vec{p}}=\discrdt{n}\chn[J]{h}{\vec{p}}$ and $\partial_t\chn{h}{\vec{d}}=\discrdt{n}\chn{h}{\vec{d}}$ a.e. in $(0,T)$, where
\[
\discrdt{n}\phi=\frac{\phi^{\apchn{n}}-\phi^{\apchn{nmm}}}{\Delta t}.
\]

\section{\Apost{} error analysis}\label{sec:apost}

Throughout this section, the notation $a\lesssim b$ means that there exists a constant $C>0$, independent of discretization parameters (but possibly dependent on data and on the final time $T$), such that $a\leq Cb$.

The \apost{} analysis that we present is based on the following properties of the continuous problem \eqref{eq:weak}, inspired by the abstract framework of \cite{ern2009posteriori}:
\begin{proposition}\label{prpstn:contcoerc}
Under the definitions of \cref{sec:weak}, and assuming that each of the Dirichlet boundaries $\Gamma_{\text{D},\star}, \star=\vec{d},\vec{u},p_\jj, \forall\jj\in J,$ is not empty, the following properties hold:
\begin{enumerate}
    \item The forms $a_\PP:\spHd\times\spHd\to\mathbb R$, $\widetilde{a}_J:\spHJ\times\spHJ\to\mathbb R$, $a_\FF:\spHu\times\spHu\to\mathbb R$, $m_J:\spLJ\times\spLJ\to\mathbb R$ are bilinear, symmetric, coercive, and continuous in their spaces of definition.
    \item The norms $\|\cdot\|_{V_\star}$ and $\|\cdot\|_{L_\star}$ are equivalent for each $\star=\vec{d},\vec{u},J,\jj\in J$.
    \item The source terms $\vec{f}_\PP,[g_\jj]_{\jj\in J}, \vec{f}_\FF$ are square-integrable and they fulfill the following inequalities:
    \[
    \|\vec{f}_\PP\|_{\Omega_\PP}^2\lesssim a_\PP(\vec{f}_\PP,\vec{f}_\PP), \quad
    \sum_{\jj\in J}\|g_\jj\|_{\Omega_\PP}^2\lesssim \widetilde{a}([g_\jj]_{\jj\in J},[g_\jj]_{\jj\in J}), \quad
    \|\vec{f}_\FF\|_{\Omega_\FF}^2\lesssim a_\FF(\vec{f}_\FF,\vec{f}_\FF).
    \]
    \item The coupling form $b_J:\spLJ\times\spHd\to\mathbb R$ is bilinear, continuous, and \[
    |b_J(\vec{f}_\PP,[g_\jj]_{\jj\in J})|\lesssim\|\vec{f}_\PP\|_{a_\PP}\|[g_\jj]_{\jj\in J}\|_{m_J}.
    \]
    \item The Stokes operator $\LLstokes([\vec{u},\vec{p}];[\test{\vec{u}},\test{p}]):=a_\FF(\vec{u},\test{\vec{u}})+b(p,\test{\vec{u}})+b(\test{p},\vec{u})$ satisfies the following inf-sup inequality: $\exists\beta_\LLstokes>0$ s.t.
    \[
    \inf_{\substack{(\vec{u},p)\in\spHu\times\spLp\\
    \vec{u}\not\equiv\vec{0},p\not\equiv 0}}
    \sup_{\substack{(\test{\vec{u}},\test{p})\in\spHu\times\spLp\\
    \test{\vec{u}}\not\equiv\vec{0},\test{p}\not\equiv 0}}
    \frac{\LLstokes([\vec{u},\vec{p}];[\test{\vec{u}},\test{p}])}
    {(\|\vec{u}\|_\spHu+\|p\|_\spLp)(\|\test{\vec{u}}\|_\spHu+\|\test{p}\|_\spLp)}
    \gtrsim \beta_\LLstokes.
    \]
\end{enumerate}
Therefore, we can consider the spaces $\spHd,\spHJ,\spLJ,\spHu$ as endowed with the norms induced by $a_\PP,\widetilde{a}_J,m_J,a_\FF$, respectively.
\\
Moreover, the properties of points 1, 2, 4, and 5 also hold at the discrete level, with a discrete inf-sup constant $\beta_{\LLstokes,h}\in(0,\beta_\LLstokes)$ in point 5.
\end{proposition}
\begin{proof}
The proof of points 1-4 is straightforwardly based on the definition of coercivity and continuity, and it exploits Korn/Poincar\'e inequalities.
Point 5 is discussed in classical works on Stokes equations, e.g. \cite{verfurth1989posteriori}.
\end{proof}

We now introduce the following consistency operators, namely~the residuals of problem \eqref{eq:discr} tested against continuous test functions:
\begin{equation}\label{eq:G}\begin{gathered}
\GG_{\vec{d}}\in\DUALspHd,\quad
\GG_{\vec{u}}\in\DUALspHu,\quad
\GG_J\in\DUALspHJ,\quad
\GG_p\in\DUALspLp=\spLp\quad
\text{defined as}
\\
    \begin{aligned}
\braket{\GG_{\vec{d}}, \test{\vec{d}}}_\spHd &=
(
\chn[\vec{d}]{h}{\vec{f}}
,\test{\vec{d}})_{\Omega_\PP}
- a_\PP(\chn{h}{\vec{d}},\test{\vec{d}}) - b_J(\chn[J]{h}{\vec{p}},\test{\vec{d}})
- \mathfrak J_\PP(\chn[\EE]{h}{p},\test{\vec{d}})
,\\
\braket{\GG_J, \test{\vec{p}}_J}_\spHJ &=
(\pi^0\chn{h}{\vec{g}},\test{\vec{p}}_J)_{\Omega_\PP}
- m_J(\partial_t \chn[J]{h}{\vec{p}},\test{\vec{p}}_J)
- \widetilde{a}_J(\pi^0\chn[J]{h}{\vec{p}},\test{\vec{p}}_J)
\\&\qquad
+ b_J(\test{\vec{p}}_J,\partial_t\chn{h}{\vec{d}})
+ \mathfrak J_\PP(\test{p}_{\EE},\partial_t\chn{h}{\vec{d}}) + \mathfrak J_\FF(\test{p}_{\EE},\chn{h}{\vec{u}})
,\\
\braket{\GG_{\vec{u}}, \test{\vec{u}}}_\spHu &=
(
\chn[\vec{u}]{h}{\vec{f}}
,\test{\vec{u}})_{\Omega_\FF}
- a_\FF(\chn{h}{\vec{u}},\test{\vec{u}})
- b_\FF(\chn{h}{p},\test{\vec{u}})
- \mathfrak J_\FF(\chn[\EE]{h}{p},\test{\vec{u}})
,\\
\GG_p &=
-\Div{\chn{h}{\vec{u}}}
,
    \end{aligned}
\end{gathered}\end{equation}
and by the usual notation, we denote by $\GG_{\vec{d}}^{\apchn{n}},\GG_J^{\apchn{n}},\GG_{\vec{u}}^{\apchn{n}},\GG_p^{\apchn{n}}$ their evaluation at time $t^{\apchn{n}}$.

We are now ready to introduce a preliminary result estimating the discretization errors in terms of the consistency operators.
\begin{theorem}
\label{th:preliminary}
Let the errors be denoted as $\err{c}{\vec{d}}=\vec{d}-\chn{h}{\vec{d}}, \err{c}{\vec{u}}=\vec{u}-\chn{h}{\vec{u}}, \err{c}{p}=p-\chn{h}{p}, \err{c}{J}=\vec{p}_J-\chn[J]{h}{\vec{p}}$ (the latter including $\err{c}{\EE}=p_\EE-\chn[\EE]{h}{p}$).
Then, under the assumptions of \cref{prpstn:contcoerc}, the following estimate holds:
\begin{equation}\label{eq:preliminaryEstimate}\begin{aligned}
    \|\err{c}{\vec{d}}\|_{\Linf{\spHd}{n}}^2
    &+ \|\err{c}{J}\|_{\Linf{\spLJ}{n}}^2 + \|\err{c}{\vec{u}}\|_{\Ltwo{\spHu}{n}}^2
    + \|\err{c}{J}\|_{\Ltwo{\spHJ}{n}}^2
    \\
    &\lesssim
    \estdata{n} + \esttime{n} + \estpre{n}
\end{aligned}\end{equation}
where
\begin{equation}\label{eq:estimatorsTh1}\begin{aligned}
    \estdata{n} &:= 
    \|\vec{d}_0-\vec{d}_{0,h}\|_\spHd^2 + \|\vec{u}_0-\vec{u}_{0,h}\|_\spHu^2 + \|\vec{p}_{J,0}-\vec{p}_{J,0,h}\|_\spHJ^2
    \\&\qquad
    + \|\vec{g}_J - \pi^0\chn[J]{h}{\vec{g}}\|_{\Ltwo{\DUALspHJ}{n}}^2
    + \|\vec{f}_{\vec{u}} - 
\chn[\vec{u}]{h}{\vec{f}}
    \|_{\Ltwo{\DUALspHu}{n}}^2
    \\&\qquad
    + \left(\|\vec{f}_{\vec{d}} - 
\chn[\vec{d}]{h}{\vec{f}}
    \|_{\Linf{\DUALspHd}{n}}
    + \|\partial_t(\vec{f}_{\vec{d}} -
\chn[\vec{d}]{h}{\vec{f}}
    )\|_{\Lone{\DUALspHd}{n}}
    \right)^2,
    \\
    \esttime{n} &:= \|\chn[J]{h}{\vec{p}}-\pi^0\chn[J]{h}{\vec{p}}\|_{\Ltwo{\spHJ}{n}}^2,
    \\
    \estpre{n} &:= \braket{\GG_{\vec{d}}^{\apchn{n}},\err{n}{\vec{d}}}_\spHd -  \braket{\GG_{\vec{d}}^0,\err{0}{\vec{d}}}_\spHd
        + \int_0^{t^{\apchn{n}}} \braket{\partial_t\GG_{\vec{d}},\err{c}{\vec{d}}}_\spHd ds
        \\&\qquad
        + \int_0^{t^{\apchn{n}}} \braket{\GG_{J},\err{c}{J}}_\spHJ ds
        + \int_0^{t^{\apchn{n}}} \braket{\GG_{\vec{u}},\err{c}{\vec{u}}}_\spHu ds
        + \int_0^{t^{\apchn{n}}} \braket{\GG_p,\err{c}{p}}_\spLp ds.
\end{aligned}\end{equation}
\end{theorem}
\begin{proof}
    Let us introduce also $\err{*}{J} = \vec{p}_J-\pi^0\chn[J]{h}{\vec{p}}$ (including its component $\err{*}{\EE} = p_\EE-\pi^0\chn[\EE]{h}{p}$).
    Subtracting the discrete problem \eqref{eq:discr} from the continuous problem \eqref{eq:weak}, both tested against generic continuous test functions $(\test{\vec{d}},\test{\vec{p}}_J,\test{\vec{u}},\test{p})\in\spHd\times\spHJ\times\spHu\times\spLp$, yields the following error problem, holding a.e.~in time:
\begin{equation}\label{eq:pberr}\left\{
    \begin{aligned}
    &a_\PP(\err{c}{\vec{d}},\test{\vec{d}}) + b_J(\err{c}{J},\test{\vec{d}})
    + a_\FF(\err{c}{\vec{u}},\test{\vec{u}}) + b_\FF(\chn{h}{p},\test{\vec{u}})
    - b_\FF(\test{p}_h,\chn{h}{\vec{u}})
    \\&\qquad\qquad
    + \mathfrak J_\PP(\err{c}{\EE},\test{\vec{d}})
    + \mathfrak J_\FF(\err{c}{\EE},\test{\vec{u}})
    \\&\qquad
    = (\vec{f}_{\vec{d}}-
\chn[\vec{d}]{h}{\vec{f}}
    ,\test{\vec{d}})_{\Omega_\PP}
    + (\vec{f}_{\vec{u}}-
\chn[\vec{u}]{h}{\vec{f}}
    ,\test{\vec{u}})_{\Omega_\FF}
    \\&\qquad\qquad
    + \braket{\GG_{\vec{d}},\test{\vec{d}}}_\spHd + \braket{\GG_{\vec{u}},\test{\vec{u}}}_\spHu
    + \braket{\GG_p,\test{p}}_\spLp,
    \\
    &m_J(\partial_t \err{c}{J},\test{\vec{p}}_{J})
    + \widetilde{a}_J(\err{*}{J},\test{\vec{p}}_{J})
    + b_J(\test{\vec{p}}_{J},\partial_t\err{c}{\vec{d}})
    - \mathfrak J_\PP(\test{p}_{\EE},\partial_t\err{c}{\vec{d}}) - \mathfrak J_\FF(\test{p}_{\EE},\err{c}{\vec{u}})
    \\&\qquad\qquad
    = (\chn{c}{\vec{g}}-\pi^0\chn{h}{\vec{g}},\test{\vec{p}}_{J})_{\Omega_\PP}
    + \braket{\GG_J,\test{\vec{p}}_J}_\spHJ.
\end{aligned}\right.\end{equation}

We now choose $\test{\vec{d}} = \partial_t\err{c}{\vec{d}}, \test{\vec{u}}=\err{c}{\vec{u}}, \test{\vec{p}}_J=\err{c}{J}, \test{p}=\err{c}{p}$ as test functions and we sum up the equations in \eqref{eq:pberr}, noticing that the terms involving the forms $b_J, b_\FF, \mathfrak{J}_\PP, \mathfrak{J}_\FF$ cancel out for this choice of the test functions.
Moreover, we observe that the following equalities hold:
\begin{gather*}
    a_\PP(\err{c}{\vec{d}},\partial_t\err{c}{\vec{d}}) = \frac{1}{2}\frac{d}{dt}a_\PP(\err{c}{\vec{d}},\err{c}{\vec{d}}),
    \qquad
    m_J(\partial_t\err{c}{J},\err{c}{J}) = \frac{1}{2}\frac{d}{dt}m_J(\err{c}{J},\err{c}{J}),
    \\
    \begin{aligned}
    \widetilde{a}_J(\err{*}{J},\err{c}{J})
    &= \frac{1}{2}\left[
    \widetilde{a}_J(\err{*}{J},\err{*}{J})
    + \widetilde{a}_J(\err{c}{J},\err{c}{J})
    - \widetilde{a}_J(\err{*}{J}-\err{c}{J},\err{*}{J}-\err{c}{J})\right]
    \\&= \frac{1}{2}\left[
    \widetilde{a}_J(\err{*}{J},\err{*}{J})
    + \widetilde{a}_J(\err{c}{J},\err{c}{J})
    - \widetilde{a}_J(\chn[J]{h}{\vec{p}}-\pi^0\chn[J]{h}{\vec{p}}, \chn[J]{h}{\vec{p}}-\pi^0\chn[J]{h}{\vec{p}})
    \right]
    \end{aligned}
\end{gather*}
due to the symmetry of the forms and the definition of $\err{*}{J}$.
Therefore, we obtain the following identity:
\begin{equation}\label{eq:proof1lhs1}\begin{aligned}
    &\frac{1}{2}\frac{d}{dt}a_\PP(\err{c}{\vec{d}},\err{c}{\vec{d}})
    +\frac{1}{2}\frac{d}{dt}m_J(\err{c}{J},\err{c}{J})
    +a_\FF(\err{c}{\vec{u}},\err{c}{\vec{u}})
    \\&\qquad\qquad
    +\frac{1}{2}
    \widetilde{a}_J(\err{*}{J},\err{*}{J})
    +\frac{1}{2}
    \widetilde{a}_J(\err{c}{J},\err{c}{J})
    \\&\qquad
    =(\vec{f}_{\vec{d}}-
\chn[\vec{d}]{h}{\vec{f}}
    ,\partial_t\err{c}{\vec{d}})_{\Omega_\PP}
    + (\vec{f}_{\vec{u}}-
\chn[\vec{u}]{h}{\vec{f}}
    ,\err{c}{\vec{u}})_{\Omega_\FF}
    + (\chn{c}{\vec{g}}-\pi^0\chn{h}{\vec{g}},\err{c}{J})_{\Omega_\PP}
    \\&\qquad\qquad
    +\frac{1}{2}\widetilde{a}_J(\chn[J]{h}{\vec{p}}-\pi^0\chn[J]{h}{\vec{p}}, \chn[J]{h}{\vec{p}}-\pi^0\chn[J]{h}{\vec{p}})
    \\&\qquad\qquad
    +\braket{\GG_{\vec{d}},\partial_t\err{c}{\vec{d}}}_\spHd + \braket{\GG_{\vec{u}},\err{c}{\vec{u}}}_\spHu
    + \braket{\GG_p,\err{c}{p}}_\spLp
    + \braket{\GG_J,\err{c}{J}}_\spHJ
    .
\end{aligned}
\end{equation}
We integrate \eqref{eq:proof1lhs1} in time from 0 to $t^{\apchn{n}}$ and we use the following integration by parts formula:
\begin{align*}
    \int_0^{t^{\apchn{n}}}
    &\left[
    (\vec{f}_{\vec{d}}-
\chn[\vec{d}]{h}{\vec{f}}
    ,\partial_t\err{c}{\vec{d}})_{\Omega_\PP}
    +\braket{\GG_{\vec{d}},\partial_t\err{c}{\vec{d}}}_\spHd
    \right]ds
    \\&\qquad
    = (\vec{f}_{\vec{d}}^{\apchn{n}}-
\chn[\vec{d}]{n}{\vec{f}}
    ,\err{n}{\vec{d}})_{\Omega_\PP}
    - (\vec{f}_{\vec{d}}^0-\vec{f}_{\vec{d},h}^0,\err{0}{\vec{d}})_{\Omega_\PP}
    +\braket{\GG_{\vec{d}}^{\apchn{n}},\err{n}{\vec{d}}}_\spHd
    -\braket{\GG_{\vec{d}}^0,\err{0}{\vec{d}}}_\spHd
    \\&\qquad\qquad
    -\int_0^{t^{\apchn{n}}}\left[
    (\partial_t(\vec{f}_{\vec{d}}-
\chn[\vec{d}]{h}{\vec{f}}
    ),\err{c}{\vec{d}})_{\Omega_\PP}
    +\braket{\partial_t\GG_{\vec{d}},\err{c}{\vec{d}}}_\spHd
    \right]ds.
\end{align*}
Then, using the coercivity of the forms $a_\PP,\widetilde{a}_J,m_J,a_\FF$ (see \cref{prpstn:contcoerc}) and employing the Cauchy-Schwarz and the Young inequalities yield the following inequality:
\begin{equation*}
\begin{aligned}
    &\|\err{n}{\vec{d}}\|_\spHd^2
    +\|\err{n}{J}\|_\spLJ^2
    +\|\err{c}{\vec{u}}\|_{\Ltwo{\spHu}{n}}^2
    +\|\err{*}{J}\|_{\Ltwo{\spHJ}{n}}^2
    +\|\err{c}{J}\|_{\Ltwo{\spHJ}{n}}^2
    \\&\quad
    \lesssim\|\err{0}{\vec{d}}\|_\spHd^2
    +\|\err{0}{J}\|_\spLJ^2
    + \|\vec{f}_{\vec{d}}^{\apchn{n}}-
\chn[\vec{d}]{n}{\vec{f}}
    \|_\spLd^2
    - \|\vec{f}_{\vec{d}}^0-\vec{f}_{\vec{d},h}^0\|_\spLd^2
    +\braket{\GG_{\vec{d}}^{\apchn{n}},\err{n}{\vec{d}}}_\spHd
    -\braket{\GG_{\vec{d}}^0,\err{0}{\vec{d}}}_\spHd
    \\&\quad\quad
    +\int_0^{t^{\apchn{n}}}\left[
    \|\partial_t(\vec{f}_{\vec{d}}-
\chn[\vec{d}]{n}{\vec{f}}
    )\|_\spLd\|\err{c}{\vec{d}}\|_\spLd
    + \|\vec{f}_{\vec{u}}-
\chn[\vec{u}]{n}{\vec{f}}
    \|_\spLu\|\err{c}{\vec{u}}\|_\spLu
    + \|\chn{c}{\vec{g}}-\pi^0\chn{h}{\vec{g}}\|_\spLJ\|\err{c}{J}\|_\spLJ
    \right]ds
    \\&\quad\quad
    +\int_0^{t^{\apchn{n}}}\left[
    \|\chn[J]{h}{\vec{p}}-\pi^0\chn[J]{h}{\vec{p}}\|_\spHJ^2
    +\braket{\partial_t\GG_{\vec{d}},\err{c}{\vec{d}}}_\spHd
    + \braket{\GG_{\vec{u}},\err{c}{\vec{u}}}_\spHu
    + \braket{\GG_p,\err{c}{p}}_\spLp
    + \braket{\GG_J,\err{c}{J}}_\spHJ
    \right]ds
    .
\end{aligned}
\end{equation*}
Using the Cauchy-Schwarz and the Young inequalities also in the time integrals yields
\begin{equation}\label{eq:endproofth1}\begin{aligned}
    \|\err{n}{\vec{d}}\|_{\spHd}^2
    &+ \|\err{n}{J}\|_{\spLJ}^2 + \|\err{c}{\vec{u}}\|_{\Ltwo{\spHu}{n}}^2 + \|\err{c}{J}\|_{\Ltwo{\spHJ}{n}}^2
    \\
    &
    + \|\vec{p}_J-\pi^0\chn[J]{h}{\vec{p}}\|_{\Ltwo{\spHJ}{n}}^2
    \lesssim
    \tildeestdata{n} + \esttime{n} + \estpre{n},
\end{aligned}\end{equation}
where
$\|\err{*}{J}\|_{\Ltwo{\spHJ}{n}}^2$ can be removed, being positive, and
\[\begin{aligned}
    \tildeestdata{n} &:= 
    \estdata{n}
    - \left(\|\vec{f}_{\vec{d}} - 
\chn[\vec{d}]{h}{\vec{f}}
    \|_{\Linf{\DUALspHd}{n}}
    + \|\partial_t(\vec{f}_{\vec{d}} - 
\chn[\vec{d}]{h}{\vec{f}}
    )\|_{\Lone{\DUALspHd}{n}}
    \right)^2
    \\&\qquad
    + \left(\|\vec{f}_{\vec{d}}^{\apchn{n}} - 
\chn[\vec{d}]{n}{\vec{f}}
    \|_{\DUALspHd}
    + \|\partial_t(\vec{f}_{\vec{d}} - 
\chn[\vec{d}]{h}{\vec{f}}
    )\|_{\Lone{\DUALspHd}{n}}
    \right)^2.
\end{aligned}\]
Since \eqref{eq:endproofth1} holds for any $n=1,\ldots,N_T$, on the left-hand side we can replace $\|\err{n}{\vec{d}}\|_{\spHd}^2 + \|\err{n}{J}\|_{\spLJ}^2$ with $\|\err{n}{\vec{d}}\|_{\Linf{\spHd}{n}}^2 + \|\err{n}{J}\|_{\Linf{\spLJ}{n}}^2$, while on the right-hand side $\tildeestdata{n}$ can be substituted by $\estdata{n}$.
This concludes the proof.
\end{proof}

We remark that in the preliminary estimate of \cref{th:preliminary}, the estimator $\estpre{n}$ is not computable \apost{} as it depends on the errors.
In the following section, we thus address its estimation in terms of local residuals and jumps.
To this aim, we introduce the jump $\jump{\vec{\phi}}_\face$ of a generic vector field $\vec{\phi}:\Omega\to\mathbb R^\dimens$ across a face $\face$, defined as follows:
\[
\jump{\vec{\phi}}_\face = \vec{\phi}^+\odot\vec{n}_\face^+ + \vec{\phi}^-\odot\vec{n}_\face^-
\qquad\text{with}\qquad
\vec{\phi}\odot\vec{n} = \frac{1}{2}(\vec{\phi}\otimes\vec{n}+\vec{n}\otimes\vec{\phi}),
\]
where the superscripts $+$ and $-$ denote the restriction on either side of the face $\face$.
This jump operator is employed in the Discontinuous Galerkin community (see, e.g., \cite{antonietti2022stability,fumagalli2024polytopal}), yet all the results still hold if the non-symmetric outer product $\vec{\phi}\otimes\vec{n}$ is used in place of $\vec{\phi}\odot\vec{n}$.

\subsection{\Apost{} estimates of the residual operators}

We now address the estimation of the terms of \cref{th:preliminary} that involve the residual operators $\GG_\star, \star=\vec{d},\vec{u},J,p$.

\begin{lemma}\label{lmm:GG}
Under the same assumptions of \cref{th:preliminary}, the following estimates hold:
\begin{enumerate}
    \item
    $\displaystyle
    \|\GG_{\vec{d}}\|_{\Linf{\DUALspHd}{T}}^2
    \lesssim
    \estim_{\vec{d}} := \sup_{0\leq n\leq N_T} \estim_{\vec{d}}^{\apchn{n}}
    $
    \qquad where
    \[\begin{aligned}
    \estim_{\vec{d}}^{\apchn{n}}
    &
    := \sum_{\elem\in\mesh{T}_\PP}\Biggl(
    h_\elem^2\|\RR[\vec{d}]{n}{\elem}\|_\elem^2
    + \sum_{\substack{\face\in\mesh{F}_\PP^\II\\\face\in\partial\elem}}
        h_\elem\|\faceRR[\vec{d}]{n}{\face}\|_\face^2
    + \sum_{\substack{\face\in\mesh{F}_\Sigma\\\face\in\partial\elem}}
        h_\elem\|\interffaceRR[\vec{d}]{n}{\face}\|_\face^2
    \Biggr),
    \\
    \RR[\vec{d}]{n}{\elem}
    &
    := \chn[\vec{d}]{n}{\vec{f}} + \Div{\sigel\left(\chn{n}{\vec{d}}|_\elem\right)} - \sum_{\jj\in J}\alpha_\jj\nabla\chn[\jj]{n}{p}|_\elem,
    \\
    \faceRR[\vec{d}]{n}{\face}
    &
    := - \jump{\sigel(\chn{n}{\vec{d}})\vec{n}_\face}_\face,
    \\
    \interffaceRR[\vec{d}]{n}{\face}
    &
    := -\sigel\left(\chn{n}{\vec{d}}\right)\vec{n}_\PP + \sum_{\jj\in J}\alpha_\jj\chn[\jj]{n}{p}\vec{n}_\PP - \chn[\EE]{n}{p}\vec{n}_\PP;
    \end{aligned}\]
    \item
    $\displaystyle
    \|\partial_t\GG_{\vec{d}}\|_{\Lone{\DUALspHd}{T}}^2
    \lesssim
    \estim_{\vec{d}}(\partial_t) := \left[\sum_{n=1}^{N_T} \Delta t\ \left(\estim_{\vec{d}}^{\apchn{n}}(\partial_t)\right)^{1/2}\right]^2
    $
    \qquad where
    \[\begin{aligned}
    \estim_{\vec{d}}^{\apchn{n}}(\partial_t)
    &
    := \sum_{\elem\in\mesh{T}_\PP}\Biggl(
    h_\elem^2\left\|\frac{\RR[\vec{d}]{n}{\elem}-\RR[\vec{d}]{nmm}{\elem}}{\Delta t}\right\|_\elem^2
    + \sum_{\substack{\face\in\mesh{F}_\PP^\II\\\face\in\partial\elem}}
        h_\elem\left\|\frac{\faceRR[\vec{d}]{n}{\face}-\faceRR[\vec{d}]{nmm}{\face}}{\Delta t}\right\|_\face^2
    \\&\qquad\qquad
    + \sum_{\substack{\face\in\mesh{F}_\Sigma\\\face\in\partial\elem}}
        h_\elem\left\|\frac{\interffaceRR[\vec{d}]{n}{\face}-\interffaceRR[\vec{d}]{nmm}{\face}}{\Delta t}\right\|_\face^2
    \Biggr);
    \end{aligned}\]
    \item
    $\displaystyle
    \|\GG_J\|_{\Ltwo{\DUALspHJ}{T}}^2
    \lesssim
    \estim_J := \sum_{n=1}^{N_T} \Delta t\ \estim_J^{\apchn{n}}
    $
    \qquad where
    \[\begin{aligned}
    \estim_J^{\apchn{n}}
    &
    := \sum_{\elem\in\mesh{T}_\PP}\Biggl(
    h_\elem^2\|\RR[J]{n}{\elem}\|_\elem^2
    + \sum_{\substack{\face\in\mesh{F}_\PP^\II\\\face\in\partial\elem}}
        h_\elem\|\faceRR[J]{n}{\face}\|_\face^2
    + \sum_{\substack{\face\in\mesh{F}_\Sigma\\\face\in\partial\elem}}
        h_\elem\|\interffaceRR[J]{n}{\face}\|_\face^2
    \Biggr),
    \\
        \RR[J]{n}{\elem}
        &
        := \begin{bmatrix}
            \RR[\jj]{n}{\elem}
        \end{bmatrix}_{\jj\in J},
        \qquad
        \faceRR[J]{n}{\face}
        := \begin{bmatrix}
            \faceRR[\jj]{n}{\face}
        \end{bmatrix}_{\jj\in J},
    \\
    \RR[\jj]{n}{\elem}
    &
    := \chn[\jj]{n}{g} - \frac{c_\jj}{\Delta t}(\chn[\jj]{n}{p}-\chn[\jj]{nmm}{p}) + \Div{\left(\frac{\kappa_\jj}{\mu_\jj}\nabla\chn[\jj]{n}{p}|_\elem\right)}
    \\&\qquad
    - \Div{\left[\frac{\alpha_\jj}{\Delta t}(\chn{n}{\vec{d}}-\chn{nmm}{\vec{d}})|_\elem\right]}
    - \beta_\jj^\text{e}\chn[\jj]{n}{p} - \sum_{\ell\in J}\beta_{\jj\ell}(\chn[\jj]{n}{p} - \chn[\ell]{n}{p}),
    \\
    \faceRR[J]{n}{\face}
    &
    := - \Jump{\frac{\kappa_\jj}{\mu_\jj}\nabla\chn[\jj]{n}{p}\cdot\vec{n}_\face}_\face,
    \\
    \interffaceRR[J]{n}{\face}
    &
    := -\sum_{\jj\in J}\left(\frac{\kappa_\jj}{\mu_\jj}\nabla\chn[\jj]{n}{p}\cdot\vec{n}_\PP\right)
    + \frac{1}{\Delta t}(\chn{n}{\vec{d}}-\chn{nmm}{\vec{d}})\cdot\vec{n}_\PP
    + \chn{n}{\vec{u}}\cdot\vec{n}_\FF;
    \end{aligned}\]
    \item
    $\displaystyle
    \|(\GG_{\vec{u}},\GG_p)\|_{\Ltwo{\DUALspHJ\times\spLp}{T}}^2
    \lesssim
    \estim_{\vec{u}p} := \sum_{n=1}^{N_T} \Delta t\ \estim_{\vec{u}p}^{\apchn{n}}
    $
    \qquad where
    \[\begin{aligned}
    \estim_{\vec{u}p}^{\apchn{n}}
    &
    := \sum_{\elem\in\mesh{T}_\FF}\Biggl(
    h_\elem^2\|\RR[\vec{u}]{n}{\elem}\|_\elem^2
    + \|\Div{\chn{n}{\vec{u}}}\|_\elem^2
    \\&\qquad\qquad
    + \sum_{\substack{\face\in\mesh{F}_\FF^\II\\\face\in\partial\elem}}
        h_\elem\|\faceRR[\vec{u}]{n}{\face}\|_\face^2
    + \sum_{\substack{\face\in\mesh{F}_\Sigma\\\face\in\partial\elem}}
        h_\elem\|\interffaceRR[\vec{u}]{n}{\face}\|_\face^2
    \Biggr),
    \\
    \RR[\vec{u}]{n}{\elem}
    &
    := \chn[\vec{u}]{n}{\vec{f}} + \Div{\tauf\left(\chn{n}{\vec{u}}|_\elem\right)} - \nabla\chn{n}{p}|_\elem,
    \\
    \faceRR[\vec{u}]{n}{\face}
    &
    := - \jump{\tauf(\chn{n}{\vec{u}})\vec{n}_\face}_\face,
    \\
    \interffaceRR[\vec{u}]{n}{\face}
    &
    := -\tauf\left(\chn{n}{\vec{u}}\right)\vec{n}_\FF + \chn{n}{p}\vec{n}_\FF - \chn[\EE]{n}{p}\vec{n}_\FF.
\end{aligned}\]
\end{enumerate}
\end{lemma}

\begin{remark}
    We point out that the local residual estimators $\RR[\star]{n}{\elem}, \star =\vec{d},J,\vec{u},$ are the local residuals of the bulk equations \eqref{eq:elasticity}-\eqref{eq:pj}-\eqref{eq:fluidMom} in strong form, while the terms $\interffaceRR[\star]{n}{\face}, \star =\vec{d},J,\vec{u},$ are arise from combinations of the interface conditions (as highlighted by the arrows):
    \[\begin{aligned}
        \interffaceRR[\vec{d}]{n}{\face}
        &
        = -\left(
            \sigel(\chn{n}{\vec{d}})\vec{n}_\PP - \sum_{\jj\in J}\alpha_\jj\chn[\jj]{n}{p}\vec{n}_\PP
            + \tauf(\chn{n}{\vec{u}})\vec{n}_\FF - \chn{n}{p}\vec{n}_\FF
            \right)
        &\qquad\longleftrightarrow \eqref{eq:BCtotalstress}\\
        &\qquad
        - \left(
            \chn[\EE]{n}{p} - \chn{n}{p} + \tauf(\chn{n}{\vec{u}})\vec{n}_\FF\cdot\vec{n}_\FF\right)\vec{n}_\PP
        &\qquad\longleftrightarrow \eqref{eq:BCnormalstress}\\
        &\qquad
        - \left(
            \tauf(\chn{n}{\vec{u}})\vec{n}_\FF\right)_\parallel,
        &\qquad\longleftrightarrow \eqref{eq:BCtgstress}\\
        \interffaceRR[J]{n}{\face}
        &
        = \sum_{\jj\in J\setminus\{\EE\}}\left(\frac{\kappa_\jj}{\mu_\jj}\nabla\chn[\jj]{n}{p}\cdot\vec{n}_\PP\right)
        &\qquad\longleftrightarrow \eqref{eq:BCpnonE}\\
        &\qquad
        + \left(\chn{n}{\vec{u}}\cdot\vec{n}_\FF + \frac{\chn{n}{\vec{d}} - \chn{nmm}{\vec{d}}}{\Delta t}\cdot\vec{n}_\PP - \frac{\kappa_\EE}{\mu_\EE}\nabla\chn[\EE]{n}{p}\cdot\vec{n}_\PP\right),
        &\qquad\longleftrightarrow \eqref{eq:BCnormalflux}\\
        \PREinterffaceRR[\vec{u}]{n}{\face}
        &
        = -\left(
            \chn[\EE]{n}{p} - \chn{n}{p} + \tauf(\chn{n}{\vec{u}})\vec{n}_\FF\cdot\vec{n}_\FF\right)\vec{n}_\PP
        &\qquad\longleftrightarrow \eqref{eq:BCnormalstress}\\
        &\qquad
        + \left(
            \tauf(\chn{n}{\vec{u}})\vec{n}_\FF\right)_\parallel.
        &\qquad\longleftrightarrow \eqref{eq:BCtgstress}
    \end{aligned}\]
\end{remark}

\begin{proof}[of \cref{lmm:GG}]
    We first observe that the following holds:
    \begin{equation}\label{eq:Gortog}\begin{gathered}
    \braket{\GG_\star,\star_h}_\spHstar = 0 \qquad\forall\star_h\in\hspHstar \qquad (\text{with }\star=\vec{d},\vec{u},\vec{p}_J)\\
    \text{and in particular, }\braket{\GG_\star,\interp{\star}{\test{\star}}}_\spHstar = 0 \qquad\forall\test{\star}\in\spHstar.
    \end{gathered}\end{equation}
    Denoting by $\GG_\star^{\apchn{n}}$ the evaluation of $\GG_\star$ at time $t^{\apchn{n}}$, for each of the points 1-4, we first estimate the operator norm $\|\GG_\star^{\apchn{n}}\|_\DUALspHstar$ a.e.~in time and then we wrap it into the Bochner space norm appearing in the thesis.
    \begin{enumerate}
        \item
        We start by estimating the numerator of \begin{equation}\label{eq:Gdualnorm}
            \|\GG_{\vec{d}}^{\apchn{n}}\|_\DUALspHd \overset{\eqref{eq:Gortog}}{=} \sup\limits_{\test{\vec{d}}\in\spHd\setminus\{0\}} \frac{\braket{\GG_{\vec{d}},\errint{\vec{d}}{\test{\vec{d}}}}_\spHd}{\|\test{\vec{d}}\|_\spHd}.
        \end{equation}

        By the definition of $\GG_{\vec{d}}$ and element-wise integration by parts (i.b.p.), we obtain
        \begin{equation}\label{eq:lmmGeqpt1}\begin{aligned}
        &\braket{\GG_{\vec{d}}^{\apchn{n}},\errint{\vec{d}}{\test{\vec{d}}}}_\spHd
        \\&\qquad
        \overset{\text{\eqref{eq:G}}}{=}
        \left(\chn[\vec{d}]{n}{\vec{f}}, \errint{\vec{d}}{\test{\vec{d}}}\right)_{\Omega_\PP}
        - a_\PP(\chn{n}{\vec{d}},\errint{\vec{d}}{\test{\vec{d}}})
        \\&\qquad\qquad
        - b_J(\chn[J]{n}{\vec{p}},\errint{\vec{d}}{\test{\vec{d}}}) - \mathfrak J_\PP(\chn[\EE]{n}{p},\errint{\vec{d}}{\test{\vec{d}}})
        \\&\qquad
        \overset{\text{(i.b.p.)}}{=}
        \sum_{\elem\in\mesh{T}_\PP}\Biggl\{
        (
        \RR[\vec{d}]{n}{\elem}
        , \errint{\vec{d}}{\test{\vec{d}}})_\elem
        \\&\qquad
        \qquad\qquad\qquad+\frac{1}{2}\sum_{\substack{\face\in\mesh{F}_\PP^\II\\\face\subset\partial\elem}}
        \left(
        -\jump{\sigel(\chn{n}{\vec{d}})\vec{n}}_\face + 
        \cancel{\sum_{\jj\in J}\alpha_\jj\jump{\chn[\jj]{n}{p}\vec{n}_\face}_\face}
        , \errint{\vec{d}}{\test{\vec{d}}}\right)_\face
        \\&\qquad
        \qquad\qquad\qquad+\sum_{\substack{\face\in\mesh{F}_\Sigma\\\face\subset\partial\elem}}
        \left(
        \interffaceRR[\vec{d}]{n}{\face}
        , \errint{\vec{d}}{\test{\vec{d}}}\right)_\face
        \Biggr\},
        \end{aligned}\end{equation}
    where the terms with $\jump{\chn[\jj]{n}{p}\vec{n}_\face}_\face$ vanish because $\hspHjj\subset C^0(\Omega_\PP)$ for all $\jj\in J$.
        Using the Cauchy-Schwarz inequality on each term and multiplying/dividing by appropriate powers of $h_\elem$ yield
        \[\begin{aligned}
        \braket{\GG_{\vec{d}}^{\apchn{n}},\errint{\vec{d}}{\test{\vec{d}}}}_\spHd
        &
        \lesssim
        \sum_{\elem\in\mesh{T}_\PP}\Biggl\{
        h_\elem\|\RR[\vec{d}]{n}{\elem}\|_\elem
        \,h_\elem^{-1} \|\errint{\vec{d}}{\test{\vec{d}}}\|_\elem
        \\&\qquad
        \qquad\qquad
        +\sum_{\substack{\face\in\mesh{F}_\PP^\II\\\face\subset\partial\elem}}
        h_\elem^{1/2}\|\faceRR[\vec{d}]{n}{\face}\|_\face
        \,h_\elem^{-1/2} \|\errint{\vec{d}}{\test{\vec{d}}}\|_\face
        \\&\qquad
        \qquad\qquad
        +\sum_{\substack{\face\in\mesh{F}_\Sigma\\\face\subset\partial\elem}}
        h_\elem^{1/2}\|\interffaceRR[\vec{d}]{n}{\face}\|_\face
        \,h_\elem^{-1/2} \|\errint{\vec{d}}{\test{\vec{d}}}\|_\face
        \Biggr\}.
        \end{aligned}\]
        Hinging upon the discrete Cauchy-Schwarz inequality applied to the sums over the elements, as well as the inequality $(a+b)^2<\frac{1}{2}(a^2+b^2)$, we can prove that
    \begin{equation}\label{eq:dEstXi}\begin{aligned}
    &\braket{\GG_{\vec{d}}^{\apchn{n}},\errint{\vec{d}}{\test{\vec{d}}}}_\spHd
    ^2
    \\
    &\qquad\lesssim
        \left[\sum_{\elem\in\mesh{T}_\PP}
            h_\elem^2\|\RR[\vec{d}]{n}{\elem}\|_\elem^2
        \right]
        \left[\sum_{\elem\in\mesh{T}_\PP}
            h_\elem^{-2}\|\errint{\vec{d}}{\test{\vec{d}}}\|_\elem^2
        \right]
    \\&\qquad\qquad
        + \Biggl[\sum_{\elem\in\mesh{T}_\PP}
        \sum_{\substack{\face\in\mesh{F}_\PP^\II\\\face\in\partial\elem}}
            h_\elem\|\faceRR[\vec{d}]{n}{\face}\|_\face^2
        \Biggr]
        \Biggl[\sum_{\elem\in\mesh{T}_\PP}
            h_\elem^{-1}\|\errint{\vec{d}}{\test{\vec{d}}}\|_
            {\partial\elem}^2
        \Biggr]
    \\&\qquad\qquad
        + \Biggl[\sum_{\elem\in\mesh{T}_\PP}
        \sum_{\substack{\face\in\mesh{F}_\Sigma\\\face\in\partial\elem}}
            h_\elem\|\interffaceRR[\vec{d}]{n}{\face}\|_\face^2
        \Biggr]
        \Biggl[\sum_{\elem\in\mesh{T}_\PP}
            h_\elem^{-1}\|\errint{\vec{d}}{\test{\vec{d}}}\|_
            {\partial\elem}^2
        \Biggr].
    \end{aligned}\end{equation}
    Then, using interpolation estimates \eqref{eq:estInterp} yields
    \[\begin{aligned}
        \braket{\GG_{\vec{d}}^{\apchn{n}},\errint{\vec{d}}{\test{\vec{d}}}}_\spHd
        &\lesssim
        \Biggl\{\sum_{\elem\in\mesh{T}_\PP}
        h_\elem^2\|\RR[\vec{d}]{n}{\elem}\|_\elem^2
        +\sum_{\substack{\face\in\mesh{F}_\PP^\II\\\face\subset\partial\elem}}
        h_\elem\|\faceRR[\vec{d}]{n}{\face}\|_\face^2
        \\&\qquad
        +\sum_{\substack{\face\in\mesh{F}_\Sigma\\\face\subset\partial\elem}}
        h_\elem\|\interffaceRR[\vec{d}]{n}{\face}\|_\face^2
        \Biggr\}^{\frac{1}{2}}
        \|\test{\vec{d}}\|_\spHd.
    \end{aligned}\]
    Now, using \eqref{eq:Gdualnorm}, we obtain $\|\GG_{\vec{d}}^{\apchn{n}}\|_\DUALspHd\lesssim\left(\estim_{\vec{d}}^{\apchn{n}}\right)^{\frac{1}{2}}$, which yields point 1 of the thesis after taking the supremum w.r.t.~time at both members and noticing that each $\estim_{\vec{d}}^{\apchn{n}}$ is constant over $(t^{\apchn{nmm}},t^{\apchn{n}})$.
    \item
    Being $\GG_{\vec{d}}$ piecewise linear in time, $\partial_t\GG_{\vec{d}}|_{(t^{\apchn{nmm}},t^{\apchn{n}})} = \discrdt{n}\GG_{\vec{d}}$ is constant on each time interval $(t^{\apchn{nmm}},t^{\apchn{n}}]$.
    Employing twice equality \eqref{eq:lmmGeqpt1}, for $\apchn{n}$ and $\apchn{nmm}$,
    yields
            \begin{equation}\label{eq:lmmGeqpt2}\begin{aligned}
        &\braket{\discrdt{n}\GG_{\vec{d}}^{\apchn{n}},\errint{\vec{d}}{\test{\vec{d}}}}_\spHd
        \\&\qquad
        =
        \sum_{\elem\in\mesh{T}_\PP}\Biggl\{
        (
        \discrdt{n}\RR[\vec{d}]{n}{\elem}
        , \errint{\vec{d}}{\test{\vec{d}}})_\elem
        \\&\qquad
        \qquad\qquad\qquad+\frac{1}{2}\sum_{\substack{\face\in\mesh{F}_\PP^\II\\\face\subset\partial\elem}}
        \left(
        \discrdt{n}\faceRR[\vec{d}]{n}{\face}
        , \errint{\vec{d}}{\test{\vec{d}}}\right)_\face
        \\&\qquad
        \qquad\qquad\qquad+\sum_{\substack{\face\in\mesh{F}_\Sigma\\\face\subset\partial\elem}}
        \left(
        \discrdt{n}\interffaceRR[\vec{d}]{n}{\face}
        , \errint{\vec{d}}{\test{\vec{d}}}\right)_\face
        \Biggr\}.
        \end{aligned}\end{equation}
    Then, proceeding in the same way as in point 1 of this proof, we can obtain
    \[
    \|\partial_t\GG_{\vec{d}}\|_\DUALspHd|_{t=t^{\apchn{n}}} \lesssim \left(\estim_{\vec{d}}^{\apchn{n}}(\partial_t)\right)^{\frac{1}{2}},
    \]
    from which summing over all the intervals $(t^{\apchn{nmm}},t^{\apchn{n}}]$ yields
    \[
    \|\partial_t\GG_{\vec{d}}\|_{\Lone{\DUALspHd}{T}} = \sum_{n=1}^{N_T} \Delta t\,\|\partial_t\GG_{\vec{d}}\|_\DUALspHd|_{t=t^{\apchn{n}}} \lesssim \sum_{n=1}^{N_T} \Delta t\,\left(\estim_{\vec{d}}^{\apchn{n}}(\partial_t)\right)^{\frac{1}{2}}.
    \]
    \item
    Recalling that, for each function $\chn{h}{\phi}$ that is piecewise linear in time,
    \linebreak[4]
    $\pi^0\chn{h}{\phi}|_{t=t^{\apchn{n}}} = \chn{n}{\phi}$ and $\partial_t\chn{h}{\phi}|_{t=t^{\apchn{n}}} = \discrdt{n}\chn{h}{\phi}$, we can proceed as in point 1 of this proof and show that
    \begin{equation}\label{eq:lmmGeqpt3}\begin{aligned}
        &\braket{\GG_J^{\apchn{n}},\errint{J}{\test{\vec{p}}_J}}_\spHJ
        \\
        &\qquad\overset{\text{\eqref{eq:G}}}{=}
        \left(\pi^0\chn[J]{h}{\vec{g}}|_{t=t^{\apchn{n}}}, \errint{J}{\test{\vec{p}}_J}\right)_{\Omega_\PP}
        - m_J(\partial_t\chn[J]{h}{\vec{p}}|_{t=t^{\apchn{n}}},\errint{J}{\test{\vec{p}}_J})
        \\
        &\qquad\qquad
        - \widetilde{a}_J(\pi^0\chn[J]{h}{\vec{p}}|_{t=t^{\apchn{n}}},\errint{J}{\test{\vec{p}}_J})
        + b_J(\errint{J}{\test{\vec{p}}_J}, \partial_t\chn{h}{\vec{d}}|_{t=t^{\apchn{n}}})
        \\
        &\qquad\qquad
        + \mathfrak J_\PP(\errint{J}{\test{\vec{p}}_J}, \partial_t\chn{h}{\vec{d}}|_{t=t^{\apchn{n}}})
        + \mathfrak J_\FF(\errint{J}{\test{\vec{p}}_J}, \chn{n}{\vec{u}})
        \\
        &\qquad\overset{\text{(i.b.p.)}}{=}
        \sum_{\elem\in\mesh{T}_\PP}\Biggl\{
        \left(\chn[J]{n}{\vec{g}}
        - M_{J,\elem}\discrdt{n}\chn[J]{h}{\vec{p}}
        - \widetilde{A}_{J,\elem}\chn[J]{n}{\vec{p}}
        + B_{\PP,\elem}\discrdt{n}\chn{h}{\vec{d}}
        , \errint{J}{\test{\vec{p}}_J}\right)_\elem
        \\
        &\qquad\qquad
        +\frac{1}{2}\sum_{\substack{\face\in\mesh{F}_\PP^\II\\\face\subset\partial\elem}}
        \left(\sum_{\jj\in J}\faceRR[\jj]{n}{\face}
        , \errint{J}{\test{\vec{p}}_J}\right)_\face
        +\sum_{\substack{\face\in\mesh{F}_\Sigma\\\face\subset\partial\elem}}
        \left(\sum_{\jj\in J}\interffaceRR[J]{n}{\face}
        , \errint{J}{\test{\vec{p}}_J}\right)_\face
        \Biggr\},
    \end{aligned}\end{equation}
    where
    \[\begin{gathered}
    M_{J,\elem}\discrdt{n}\chn[J]{h}{\vec{p}} = \begin{bmatrix}
    c_\jj\discrdt{n}\chn[\jj]{h}{p}
    \end{bmatrix}_{\jj\in J},
    \qquad
    B_{\PP,\elem}\discrdt{n}\chn{h}{\vec{d}} = \begin{bmatrix}
    -\Div{(\alpha_\jj\discrdt{n}\chn{h}{\vec{d}}|_\elem)}
    \end{bmatrix}_{\jj\in J},
    \\
    \widetilde{A}_{J,\elem}\chn[J]{n}{\vec{p}} = \begin{bmatrix}
    -\Div{\left(\frac{\kappa_\jj}{\mu_\jj}\nabla\chn[\jj]{n}{p}|_\elem\right)}
    + \beta_\jj^\text{e}\chn[\jj]{n}{p} + \sum_{\ell\in J}\beta_{\jj\ell}(\chn[\jj]{n}{p} - \chn[\ell]{n}{p})
    \end{bmatrix}_{\jj\in J}.
    \end{gathered}\]
    Then, we employ the definition of $\|\GG_J^{\apchn{n}}\|_{\DUALspHJ}$ on the left-hand side of \eqref{eq:lmmGeqpt3} and the interpolation estimates \eqref{eq:estInterp} together with the Cauchy-Schwarz inequality on the right-hand side, thus obtaining $\|\GG_J^{\apchn{n}}\|_\DUALspHJ^2 \lesssim \estim_J^{\apchn{n}}$.
    Again, the second member of this inequality is piecewise constant in time, therefore we can obtain the desired result.
    \item
   Proceeding as in point 1, by definition \eqref{eq:G} of $\GG_{\vec{u}}$ and $\GG_p$, and using the continuous problem \eqref{eq:weak-stokesMom}-\eqref{eq:weak-stokesCont}, we have
    \begin{equation}\label{eq:lmmGeqpt4}\begin{aligned}
        &\braket{(\GG_{\vec{u}}^{\apchn{n}},\GG_p^{\apchn{n}}),(\errint{\vec{u}}{\test{\vec{u}}}, \test{p})}_{\spHu\times\spLp}
        \\
        &\quad\overset{\text{\eqref{eq:G}}}{=}
        \left(\chn[\vec{u}]{n}{\vec{f}}, \errint{\vec{u}}{\test{\vec{u}}}\right)_{\Omega_\FF}
        - \LLstokes([\chn{n}{\vec{u}}, \chn{n}{p}],[\errint{\vec{u}}{\test{\vec{u}}}, \test{p}])
        \\
        &\qquad
        - \mathfrak J_\FF(\chn[\EE]{n}{p}, \errint{\vec{u}}{\test{\vec{u}}})
        \\
        &\overset{\text{(i.b.p.)}}{=}
        \sum_{\elem\in\mesh{T}_\FF}\Biggl\{
        (
        \RR[\vec{u}]{n}{\elem}
        , \errint{\vec{u}}{\test{\vec{u}}})_\elem
        - (\Div{\chn{n}{\vec{u}}}, \test{p})_\elem
        \\&\qquad
        \qquad\qquad+\frac{1}{2}\sum_{\substack{\face\in\mesh{F}_\FF^\II\\\face\subset\partial\elem}}
        \left(
        -\jump{\tauf(\chn{n}{\vec{u}})\vec{n}}_\face + \cancel{\jump{\chn{n}{p}\vec{n}_\face}_\face}
        , \errint{\vec{u}}{\test{\vec{u}}}\right)_\face
        \\&\qquad
        \qquad\qquad+\sum_{\substack{\face\in\mesh{F}_\Sigma\\\face\subset\partial\elem}}
        \left(
        \interffaceRR[\vec{u}]{n}{\face}
        , \errint{\vec{u}}{\test{\vec{u}}}\right)_\face
        \Biggr\}
    \end{aligned}\end{equation}
    where $\jump{\chn{n}{p}\vec{n}_\face}_\face$ vanishes because $\hspLp\subset C^0(\Omega_\FF)$.
    Now, as in point 1, we use the definition of $\|(\GG_{\vec{u}}^{\apchn{n}},\GG_p^{\apchn{n}})\|_{\DUALspHu\times\spLp}$ on the left-hand side of \eqref{eq:lmmGeqpt4}, we apply integral and discrete Cauchy-Schwarz inequalities to the right-hand side, as well as the interpolation estimates for the terms tested against $\errint{\vec{u}}{\test{\vec{u}}}$, thus obtaining 
    \[\begin{aligned}
        &\braket{(\GG_{\vec{u}}^{\apchn{n}},\GG_p^{\apchn{n}}),(\errint{\vec{u}}{\test{\vec{u}}}, \test{p})}_{\spHu\times\spLp}
        \\&\qquad
        \lesssim
        \Biggl\{\sum_{\elem\in\mesh{T}_\FF}
        h_\elem^2\|\RR[\vec{u}]{n}{\elem}\|_\elem^2
        +\|\Div{\chn{n}{\vec{u}}}\|_\elem^2
        \\&\qquad\quad
        +\sum_{\substack{\face\in\mesh{F}_\FF^\II\\\face\subset\partial\elem}}
        h_\elem\|\faceRR[\vec{u}]{n}{\face}\|_\face^2
        +\sum_{\substack{\face\in\mesh{F}_\Sigma\\\face\subset\partial\elem}}
        h_\elem\|\interffaceRR[\vec{u}]{n}{\face}\|_\face^2
        \Biggr\}^{\frac{1}{2}}
        \|(\test{\vec{u}},\test{p})\|_{\spHu\times\spLp}.
    \end{aligned}\]
    Therefore, $\|(\GG_{\vec{u}}^{\apchn{n}},\GG_p^{\apchn{n}})\|_{\DUALspHJ\times\spLp}^2
    \lesssim \estim_{\vec{u}p}^{\apchn{n}}$
    and we conclude the proof by integrating both sides of the inequality w.r.t.~time and noticing that the right-hand side is piecewise constant.
\end{enumerate}
\end{proof}

\subsection{\Apost{} error estimates}

We are now ready to combine the results of the previous sections to prove the following \apost{} error estimate:
\begin{theorem}
\label{th:final}
    Under the assumptions of \cref{th:preliminary}, the following \apost{} error estimates hold:
    \begin{equation}\label{eq:aposteriori}\begin{aligned}
    \|\err{c}{\vec{d}}   &   \|_{\Linf{\spHd}{T}}^2
    + \|\err{c}{J}\|_{\Linf{\spLJ}{T}}^2 + \|\err{c}{\vec{u}}\|_{\Ltwo{\spHu}{T}}^2 + \|\err{c}{J}\|_{\Ltwo{\spHJ}{T}}^2 
    \\
    &    + \|\errRR{c}{\vec{d}}(\partial_t)\|_{\Lone{\DUALspHd}{T}}^2 + \|\errRR{c}{J}\|_{\Ltwo{\DUALspHJ}{T}}^2 + \|\errRR{c}{\vec{u}}\|_{\Ltwo{\DUALspHu}{T}}^2
    \\&
    + \|\Div{\err{c}{\vec{u}}}\|_{\Ltwo{\spLp}{T}}^2
    \lesssim
    \estdata{T} + \esttime{T} + \estok{T},
    \end{aligned}\end{equation}
    where
    \[\begin{aligned}
    \errRR{c}{\vec{d}}(\partial_t)
        &
        = A_{\vec{d}}\partial_t\err{c}{\vec{d}} + B_\PP^*\partial_t\err{c}{J}
        := -\Div{\sigel(\partial_t\err{c}{\vec{d}})} + \sum_{\jj\in J}\alpha_\jj\nabla\partial_t\err{c}{J},
    \\
    \errRR{c}{J}
        &
        = M_J\partial_t\err{c}{J} + B_\PP\partial_t\err{c}{\vec{d}} + A_J\err{c}{J} + \widetilde{A}_J\err{c}{J}
        \\ &
        := \begin{bmatrix}
        \displaystyle
        c_\jj\partial_t\err{c}{\jj} + \Div{(\alpha_\jj\partial_t\err{c}{\vec{d}})} - \Div{\left(\frac{\kappa_\jj}{\mu_\jj}\nabla\err{c}{J}\right)} + \left(\beta_\jj^\text{e}\err{c}{\jj} + \sum_{\ell\in J}\beta_{\jj\ell}(\err{c}{\jj}-\err{c}{\ell})\right)
        \end{bmatrix}_{\jj\in J},
    \\
    \errRR{c}{\vec{u}}
        &
        = A_{\vec{u}}\err{c}{\vec{u}} + B_\FF^*\err{c}{p}
        := -\Div{\tauf(\err{c}{\vec{u}})} + \nabla\err{c}{p},
    \\
    \estok{T}
        &
        := \estim_{\vec{d}}^{N_T} + \estim_{\vec{d}}(\partial_t)^{N_T} + \estim_J^{N_T} + \estim_{\vec{u}p}^{N_T}
        \qquad\text{(with each term defined in \cref{lmm:GG})}
    \\
    \text{and }&\estdata{T}\text{ and }\esttime{T}\text{ are the same of \eqref{eq:estimatorsTh1}.}
    \end{aligned}\]
\end{theorem}
\begin{proof}
We proceed taking inspiration from \cite{ern2009posteriori,eliseussen2023posteriori,verfurth1989posteriori}, in which inequalities like the one in \cref{th:preliminary} are combined with \apost{} estimates of the consistency operators like those of \cref{lmm:GG}.
    We start observing that \Cref{th:preliminary} implies
    \[\begin{aligned}
    \|\err{c}{\vec{d}}   &   \|_{\Linf{\spHd}{T}}^2
    + \|\err{c}{J}\|_{\Linf{\spLJ}{T}}^2 + \|\err{c}{\vec{u}}\|_{\Ltwo{\spHu}{T}}^2 + \|\err{c}{J}\|_{\Ltwo{\spHJ}{T}}^2 
    \\
    &\quad
    + \|\errRR{c}{\vec{d}}(\partial_t)\|_{\Lone{\DUALspHd}{T}}^2 + \|\errRR{c}{J}\|_{\Ltwo{\DUALspHJ}{T}}^2 + \|\errRR{c}{\vec{u}}\|_{\Ltwo{\DUALspHu}{T}}^2 + \|\Div{\err{c}{\vec{u}}}\|_{\Ltwo{\spLp}{T}}^2
    \\
    &\lesssim
    \estdata{T} + \esttime{T} + \estpre{T}
    \\
    &\quad
    + \|\errRR{c}{\vec{d}}(\partial_t)\|_{\Lone{\DUALspHd}{T}}^2 + \|\errRR{c}{J}\|_{\Ltwo{\DUALspHJ}{T}}^2 + \|\errRR{c}{\vec{u}}\|_{\Ltwo{\DUALspHu}{T}}^2 + \|\Div{\err{c}{\vec{u}}}\|_{\Ltwo{\spLp}{T}}^2.
    \end{aligned}\]
    Moreover, a combination of \cref{lmm:GG} and the Young inequality yields $\estpre{T}\lesssim\estok{T}$.
    To estimate the terms involving $\errRR{c}{\star},\star=\vec{d},J,\vec{u}$ and $\Div{\err{c}{\vec{u}}}$, we rely on the following error equations in operator form (holding a.e.~in time), that can be obtained by a subtraction of the discrete problem \eqref{eq:discr} from the continuous one \eqref{eq:weak}:
    \begin{subequations}\begin{align}
        A_{\vec{d}}\err{c}{\vec{d}} + B_\PP^*\err{c}{J}
        &= \vec{f}_{\vec{d}} - \chn[\vec{d}]{h}{\vec{f}} + \GG_{\vec{d}}
            & \text{in }\DUALspHd,    \label{eq:proofThFinalPRERd}
        \\
        \errRR{c}{J} = M_J\partial_t\err{c}{J} + B_\PP\partial_t\err{c}{\vec{d}} + A_J\err{c}{J} + \widetilde{A}_J\err{c}{J}
        &= \vec{g}_J - \pi^0\chn[J]{h}{\vec{g}} + \GG_J
            & \text{in }\DUALspHJ,    \label{eq:proofThFinalRJ}
        \\
        \errRR{c}{\vec{u}} = A_{\vec{u}}\err{c}{\vec{u}} + B_\FF^*\err{c}{p}
        &= \vec{f}_{\vec{u}} - \chn[\vec{u}]{h}{\vec{f}} + \GG_{\vec{u}}
            & \text{in }\DUALspHu,    \label{eq:proofThFinalRu}
        \\
        \Div{\err{c}{\vec{u}}}
        &= \GG_p
            & \text{in }\spLp.    \label{eq:proofThFinalDiv}
    \end{align}\end{subequations}
    Now, computing the $L^2$ norm in time of both members of \eqref{eq:proofThFinalRJ}-\eqref{eq:proofThFinalRu}-\eqref{eq:proofThFinalDiv} and applying \cref{lmm:GG}, as well as recalling the definition \eqref{eq:estimatorsTh1} of $\estdata{n}$, we can prove the following estimates:
    \begin{align*}
        \|\errRR{c}{J}\|_{\Ltwo{\DUALspHJ}{T}}^2
        &\lesssim \estdata{T} + \estim_J,
        \\
        \|\errRR{c}{\vec{u}}\|_{\Ltwo{\DUALspHu}{T}}^2
        &\lesssim \estdata{T} + \estim_{\vec{u}p},
        \\
        \|\Div{\err{c}{\vec{u}}}\|_{\Ltwo{\spLp}{T}}^2
        &\lesssim \estim_{\vec{u}p}.
    \end{align*}
    Finally, to estimate $\|\errRR{c}{\vec{d}}(\partial_t)\|_{\Lone{\DUALspHd}{T}}$, we apply the time derivative $\partial_t$ to both sides of \eqref{eq:proofThFinalPRERd}, we compute the Bochner $L^1$-norm, and we apply again \cref{lmm:GG}:
    \begin{align*}
        \|\errRR{c}{\vec{d}}(\partial_t)\|_{\Lone{\DUALspHd}{T}}^2
        &= 
        \|\partial_t(A_{\vec{d}}\err{c}{\vec{d}} + B_\PP^*\err{c}{J})\|_{\Lone{\DUALspHd}{T}}^2
        \\&
        \lesssim
        \|\partial_t(\vec{f}_{\vec{d}} - \chn[\vec{d}]{h}{\vec{f}})\|_{\Lone{\DUALspHd}{T}}^2
            + \|\partial_t\GG_{\vec{d}}\|_{\Lone{\DUALspHd}{T}}^2
        \\
        &\lesssim \estdata{T} + \estim_{\vec{d}}(\partial_t).
    \end{align*}
    This concludes the proof.
\end{proof}

\section{Numerical results}\label{sec:results}

Bearing in mind the application that motivates the present work, namely brain multiphysics flow modeling, the geometrical complexity of the domain asks for  adaptive mesh refinement strategies, whereas the benefit of adaptivity in time would be limited by the relatively slow and regular flow regime \cite{baledent2001cerebrospinal,chou2016fully,guo2018subject}.
For this reason, in this section we verify the estimates of \cref{th:final} and assess the efficiency of the estimators with respect to $h$ refinement, whereas $\Delta t$ is chosen sufficiently small not to hinder the convergence order.

We consider problem \eqref{eq:NSMPE} in $\dimens=2$ dimensions with 1 compartment $J=\{\EE\}$.
In the domain $\Omega=\Omega_\PP\cup\Omega_\FF=(-0.5,0)\times(0,0.5)\cup(0,0.5)\times(0,0.5)$, with interface $\Sigma=\{0\}\times(0,0.5)$, the following
is a solution of \eqref{eq:NSMPE} for proper values of the source functions $\vec{f}_\PP, \vec{g}_\EE,\vec{f}_\FF$:
\begin{align*}
\vec{d}(t,\vec{x}) &= \left(\cos(\eta t)-\sin(\eta t)\right)
\pi\frac{\kappa}{\eta}
\left(
    \sin(\pi x)\sin(\pi y) - \cos(\pi x)\cos(\pi y)
\right)
\begin{bmatrix}
1 \\ -1
\end{bmatrix},
\\
p_\EE(t,\vec{x}) &= -\left(\cos(\eta t)-\sin(\eta t)\right)\left(\pi x\cos(\pi y)+2\pi^2\mu\frac{\kappa}{\mu_\PP}\sin(\pi y)\right),
\\
\vec{u}(t,\vec{x}) &= 2\cos(\eta t)
\pi\frac{\kappa}{\mu_\PP}
\left(
    \sin(\pi x)\sin(\pi y) - \cos(\pi x)\cos(\pi y)
\right)
\begin{bmatrix}
-1 \\ 1
\end{bmatrix},
\\
p(t,\vec{x}) &= -\left(1.5\cos(\eta t)-0.5\sin(\eta t)\right)\left(x\cos(\pi y)+4\pi^2\mu\frac{\kappa}{\mu_\PP}\sin(\pi y)\right),
\end{align*}
with $\eta = \frac{\mu_\PP}{\mu_\FF(1-\alpha)}$.
In particular, we enforce fully Dirichlet boundary conditions on $\partial\Omega$, namely $\Gamma_{\text{N},\vec{d}}=\Gamma_{\text{N},J}=\Gamma_{\text{N},\vec{u}}=\emptyset$.
We simulate the system for $T=\SI{5e-7}{\second}$ with $\Delta t=\SI{1e-7}{\second}$ and choose $s=1$ in the finite element spaces \eqref{eq:fe}, namely quadratic elements for $\vec{d},\vec{u},p_\EE$ and linear elements for the Stokes pressure $p$: this choice ensures that the properties of \cref{prpstn:contcoerc} hold true.
We set all physical parameters equal to 1,
except for $\alpha_\EE=0.5$.

In the following, we discuss the results for the error energy norm
\begin{equation}\label{eq:erresults}\begin{aligned}
\ERRnoR&:=\|\err{c}{\vec{d}}\|_{\Linf{\spHd}{T}}^2 + \|\err{c}{J}\|_{\Linf{\spLJ}{T}}^2 + \|\err{c}{\vec{u}}\|_{\Ltwo{\spHu}{T}}^2 + \|\err{c}{J}\|_{\Ltwo{\spHJ}{T}}^2
,
\end{aligned}\end{equation}
which includes all the physically relevant terms of \eqref{eq:aposteriori}.
For simplicity, we decide to neglect the term $\estdata{T}$ from the \apost{} error estimate \eqref{eq:aposteriori}.

\begin{figure}
    \centering
    \includegraphics[width=0.48\textwidth]
    {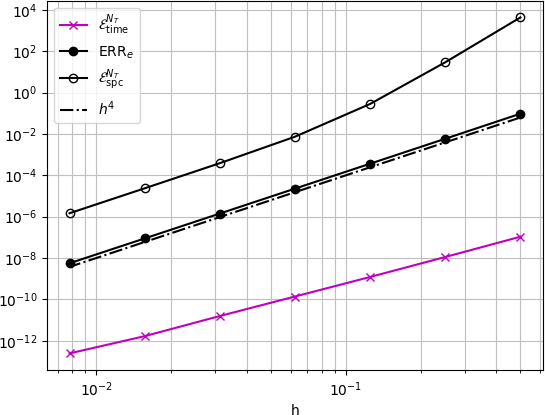}
    \hfill
    \includegraphics[width=0.48\textwidth]
    {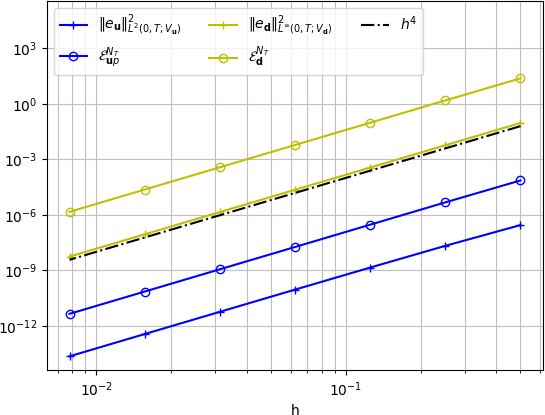}\\
    \caption{Convergence test w.r.t.~space discretization parameter $h$. Left: energy error norm $\ERRnoR$ and estimators $\protect\estok{T},\protect\esttime{T}$.
    Right: terms in the definitions of $\ERRnoR$ and $\protect\estok{T}$ (see \eqref{eq:erresults} and \cref{th:final}).
    }
    \label{fig:conv}
\end{figure}

In \cref{fig:conv} (left) we report the results of a convergence test for the error norms \eqref{eq:erresults} w.r.t.~$h$ and the corresponding values of the estimator $\estok{T}$.
The estimate \eqref{eq:aposteriori} of \cref{th:final} is verified by observing that $\ERRnoR\leq\estok{T}$ for all values of $h$.
Moreover, we notice the significantly small values obtained for the estimator $\esttime{T}$, due to the choice of a small value for $\Delta t$.
In \cref{fig:conv} (right) we analyze the contribution of the different terms entering in the space error estimator $\estok{T}$.
We notice that $\estim_{\vec{u}p}^{N_T}$ and $\estim_{\vec{d}}^{N_T}$ seem to be reliable estimators of $\|\err{c}{\vec{u}}\|_{\Ltwo{\spHu}{T}}^2$ and $\|\err{c}{\vec{d}}\|_{\Linf{\spHd}{T}}^2$, respectively.

\begin{figure}
    \centering
    \includegraphics[width=0.48\textwidth]
    {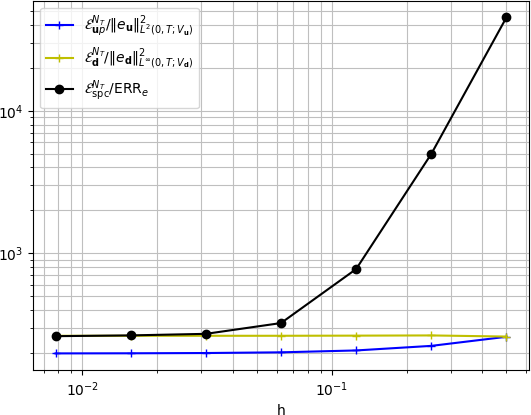}
    \caption{Efficiency indexes for the estimator $\protect\estok{T}$ and of its constituting terms, in the convergence test w.r.t.~$h$.
    }
    \label{fig:eff}
\end{figure}

To analyze the efficiency of the estimator, in \cref{fig:eff} we display the efficiency index $I_\text{eff}=\frac{\estok{T}}{\ERRnoR}$ and for completeness, even though not covered by our theory, we also report the analogous ratios for the single components of the error and estimator discussed above.
First, we notice that the estimator $\estok{T}$ is shown to be efficient for sufficiently small values of $h$. Similarly,  $\estim_{\vec{u}p}^{N_T}$ and $\estim_{\vec{d}}^{N_T}$ appear to be efficient estimators of the Stokes and elastic displacement errors, respectively.

\section{Conclusions}

In the present study, we have rigorously derived -- for the first time -- residual-based \apost{} error estimates for the finite element discretization of the coupled Stokes-MPE system.
The associated estimator controls both the error in the energy norm of the system and additional terms related to the strong residual of the equations, in dual norms.
We have verified the reliability and efficiency of the estimator by means of numerical experiments in a case with a manufactured solution.

The present work represents a first step towards the design of adaptive refinement algorithms and reduced order models for the efficient solution of fluid-poromechanics problems.
This is particularly relevant in the study of brain multiphysics flow, in which the complexity of the geometry and fluid-tissue interface may hinder the feasibility of numerical simulations under limited resources, and in multi-query problems like model calibration and validation against clinical data \cite{baledent2001cerebrospinal,pahlavian2018regional}.
Further extensions of the analysis presented in this work may also be considered in the case of advanced numerical methods based on general mesh elements, such as polytopal discontinuous Galerkin and virtual element methods, in which the geometrical flexibility of the numerical scheme makes it particularly suitable for local refinement strategies \cite{antonietti2016review,bassi2012flexibility,beirao2023adaptive,cangiani2017posteriori,canuto2023posteriori,fumagalli2024discontinuous}.

\section*{Acknowledgments}
IF and NP have been partially supported by ICSC–Centro Nazionale di Ricerca in High Performance Computing, Big Data, e Quantum Computing funded by the European Union--NextGenerationEU plan. 
MV has been partially funded by the European Union (ERC SyG, NEMESIS, project number 101115663).
The present research is part of the activities of ``Dipartimento di Eccellenza 2023-2027'', Dipartimento di Matematica, Politecnico di Milano.
All the authors are members of GNCS-INdAM
and IF acknowledges the support of the GNCS project CUP E53C23001670001.

\bibliography{main}

\end{document}